\documentclass[11pt, a4paper]{amsart} 
\usepackage[T2A]{fontenc}
\usepackage[utf8]{inputenc}
\usepackage{geometry}
\usepackage{color}
\usepackage{amsmath, amssymb, amsfonts, amsthm}
\usepackage{enumitem}
\usepackage{graphicx}
\usepackage{tikz}
\usepackage{hyperref}
\usepackage{bbm}

\usepackage{comment}
\usepackage[utf8]{inputenc}
\usepackage[T2A]{fontenc}
\usepackage[english]{babel}
\renewcommand\atop[2]{\genfrac{}{}{0pt}{}{#1}{#2}}

\newtheorem{theorem}{Theorem}

\newtheorem{lemma}{Lemma}
\newtheorem{prop}{Proposition}
\newtheorem{defin}{Definition}
\theoremstyle{remark}
\newtheorem*{rem}{Remark}

\newenvironment{proof}{\noindent{\bf Proof:}}{$\hfill \Box$ \vspace{10pt}}  

\def\XXint#1#2#3{{\setbox0=\hbox{$#1{#2#3}{\int}$ }
\vcenter{\hbox{$#2#3$ }}\kern-.6\wd0}}

\newcommand{\abs}[1]{\left|#1\right|}
\newcommand{\norm}[1]{\left\|#1\right\|}

\newcommand{\with}{\quad\!\hbox{with}\!\quad}
\newcommand{\andf}{\quad\!\hbox{and}\!\quad}
\newcommand{\R}{\mathbb{R}}
\newcommand{\N}{\mathbb{N}}

\newcommand\1{\mathbbm{1}}
\newcommand{\cC}{\mathcal{C}}
\newcommand{\cS}{\mathcal{S}}

\renewcommand{\div}{{\rm div}\,}
\newcommand{\BMO}{\mathrm{BMO}^{-1}}

\begin{document}
\title[Density-dependent incompressible NS equations in tent spaces]{Density-dependent incompressible Navier--Stokes equations in critical tent spaces}
\author{R. Danchin, I. Vasilyev}
\address{Univ Paris Est Creteil, Univ Gustave Eiffel, CNRS, LAMA UMR8050, F-94010 Creteil, France
and Sorbonne Universit\'e, LJLL UMR 7598, 4 Place Jussieu, 75005 Paris, France
}
\address{St.-Petersburg Department of V.A. Steklov Mathematical Institute, Russian Academy of Sciences (PDMI RAS), Fontanka 27, St.-Petersburg, 191023, Russia}
\email{danchin@u-pec.fr}
\email{ivasilyev@pdmi.ras.ru}
\begin{abstract}
In this paper we prove the existence of global solutions to the inhomogeneous incompressible Navier--Stokes equations, whenever the initial velocity is small enough in some subspace of $\mathrm{BMO}^{-1}$ and the initial density is sufficiently close to $1$ in the uniform metric. This is a natural extension to the variable density
case of the celebrated result by H. Koch and D. Tataru
\cite{kochtataru} concerning the classical Navier--Stokes equations.
\end{abstract}

%% Beginn des Briefes %%%%%%%%%%%%%%%%%%%%%%%%%%%%%%%%%%%%%%%
\maketitle
\section*{Introduction}
The present paper is concerned with the construction of global-in-time solutions to 
 the following  inhomogeneous (or density-dependent) incompressible Navier--Stokes equations
\begin{equation}\label{inhom}
\begin{cases}
\partial_t(\rho u) + \div(\rho u \otimes u)  -\Delta u + \nabla P=0\qquad&\hbox{in }\ \R^3\times \R_+\\\
\partial_t \rho +\div(\rho u)=0\qquad&\hbox{in }\  \R^3\times \R_+\\
\mathrm{div}\, u=0\qquad&\hbox{in }\ \R^3\times \R_+
%\\u|_{t=0} = u_0, \quad \rho|_{t=0} =  \rho_0 \qquad&\hbox{in }\ \R^{3}
\end{cases}
\end{equation}
supplemented with initial data $(\rho_0,u_0)$ at time $t=0.$
\smallbreak
These equations govern the evolution of the density $\rho=\rho(x,t) \geqslant 0,$ the velocity $u=u(x,t)\in\R^3$ and the pressure 
$P=P(x,t)\in\R$  of viscous incompressible flows with variable density in the whole space $\R^3.$
They are  a toy model for describing  mixtures of incompressible homogeneous fluids (in which case it is natural to consider discontinuous density), or pollutants.
\smallbreak
Since the pioneering work by J. Leray \cite{Leray} in 1934, a huge mathematical literature has been dedicated to 
 the case with constant density, namely 
\begin{equation}\label{ns}
\begin{cases}
\partial_t u + \div(u \otimes u)  -\Delta u + \nabla P=0\qquad&\hbox{in }\   \R^3\times \R_+\\
\mathrm{div}\,u =0\qquad&\hbox{in }\  \R^3\times \R_+\\
u|_{t=0} = u_0\qquad&\hbox{in }\ \R^{3}.
\end{cases}
\end{equation}
J. Leray established that any divergence-free square integrable 
initial velocity field $u_0$ generates at least one
 global-in-time distributional solution of \eqref{ns} (that he called \emph{turbulent solution}) satisfying the energy inequality:
 $$\frac12\|u(t)\|_{L^2}^2+\int_0^t\|\nabla u(\tau)\|_{L^2}^2\,d\tau\leqslant \frac12\|u_0\|_{L^2}^2,\qquad t\geqslant 0.$$
He also pointed out that  smoother data give rise to smoother global and unique solutions  if 
$\|u_0\|_{L^2}\|\nabla u_0\|_{L^2}$ is small enough.  At the same time, he left open the question of global existence of a smooth solution in the case of large data.
\medbreak
In 1964, H. Fujita and T. Kato \cite{FK} observed that solving \eqref{ns} may be reformulated in terms of a fix point problem, namely it is formally equivalent to
\begin{equation}\label{nsfix}
u(t)=e^{tA} u_0
-\int_0^te^{(t-\tau)A}(u\cdot\nabla u)(\tau)\,d\tau,\qquad t>0,
\end{equation}
   where  $(e^{tA})_{t>0}$ denotes the Stokes semi-group.
   \smallbreak   
        In the case of three dimensional domains, by taking advantage of the classical fixed point theorem in complete metric spaces, they established the existence and uniqueness of a unique global-in-time solution whenever the (divergence-free) initial velocity $u_0$ has $H^{\frac{1}{2}}$ Sobolev regularity. It is by now well understood that when the fluid domain is the whole space, the optimal framework to solve \eqref{ns} by means of the fixed point theorem 
        must be sought among the Banach spaces $X$ with the same scaling invariance as \eqref{ns}, that is, satisfying for all $\lambda>0,$
        \begin{equation}\label{eq:critical}
        \|u_\lambda\|_X\simeq \|u\|_X\with u_\lambda(x,t):=\lambda u(\lambda x,\lambda^2t).\end{equation}
        Thus, a\emph{critical space} $X_0$ for initial data must have the scaling invariance
          \begin{equation}\label{eq:critical0}
        \|u_0(\lambda\cdot)\|_{X_0}\simeq \|u_0\|_{X_0}.\end{equation}
        This approach based on the contractive fixed point theorem
        gives local-in-time well-posedness results for large 
        divergence-free initial velocities  with critical regularity (that is, belonging to a space $X_0$ satisfying
        \eqref{eq:critical0}), and global  well-posedness if $\|u_0\|_{X_0}$ is   small enough. 
After Fujita and Kato’s article, a number of critical spaces in which the Navier-Stokes equations can be solved by this method 
have been discovered, in particular all homogeneous Besov spaces $\dot B^{-1+\frac{3}{p}}_{p,r}$ with $p<\infty$
(see e.g. \cite[Chap. 5]{BCD}).

  The largest critical  Banach space continuously embedded in the set of tempered distributions that is translation invariant is the Besov space $\dot B^{-1}_{\infty,\infty}$
 (see e.g. \cite[Prop. 5.31]{BCD}), and one can wonder if \eqref{ns} is well-posed for any $u_0$ in it. 
 In 2008,  J. Bourgain and N. Pavlovic answered in the negative: the Navier-Stokes equations
 are ill-posed in $\dot B^{-1}_{\infty,\infty}$ \emph{in the sense of norm inflation}, meaning that there exist
 smooth data, that can be arbitrarily small in $\dot B^{-1}_{\infty,\infty}$  and for which 
 $\|u(t)\|_{\dot B^{-1}_{\infty,\infty}}$ is   arbitrarily large  after an arbitrarily short time (see \cite{BP} and also \cite{check}). 
Before that, H.~Koch and D.~Tataru showed in \cite{kochtataru} that 
 \eqref{nsfix} may be solved by fixed point for small enough divergence-free $u_0$ with components in the space 
 \begin{equation}\label{eq:BMO-1}
  \BMO:= \Bigl\{f\in\cS'(\R^3)\: : \:  \underset{\atop{r>0}{x_0 \in \mathbb R^3}}{\sup} \, \Bigl( \frac{1}{r^{3}} \int_0 ^{r^2}\!\! \int_{B(x_0, r)} 
  |e^{t\Delta} f(y)|^2 \, dy\,dt \Bigr)^{\frac {1}{2}} < \infty\Bigr\},
\end{equation}
with $(e^{t\Delta})_{t>0}$ denoting the heat semi-group.
\smallbreak
The space  $\BMO$ contains all the critical  homogeneous Besov spaces $\dot B^{-1+\frac{3}{p}}_{p,r}$ with $1\leqslant p<\infty,$ $1\leqslant r\leqslant\infty$
and is, so far, the largest known functional space for which the fixed point method may be implemented for solving
\eqref{ns}. 
\medbreak
The present work aims to address the question of solvability of \emph{inhomogeneous} Navier--Stokes equations \eqref{inhom}
in the case where the initial velocity belongs to a critical space in the spirit of $\BMO$ and  the initial density is  close to $1$ in $L^\infty.$
This  assumption on the data is in accordance with the scaling invariance of the equations, namely
\begin{equation}\label{eq:scaling}
(\rho,u,P)(x,t)\leadsto (\rho,\lambda u,\lambda^2P)(\lambda x,\lambda^2t),\qquad \lambda>0.
\end{equation}
We know many critical spaces in which we can solve \eqref{inhom} for all time if the data are small. For example, we have:
\begin{itemize}
\item the well-posedness for $(\rho_0-1,u_0)\in  \dot B^{\frac{3}{2}}_{2,1}\times \dot B^{\frac{1}{2}}_{2,1}$ (see \cite{dan1});
\item more generally, the well-posedness if 
$(\rho_0-1,u_0)\in \dot B^{\frac{3}{p}}_{p,1}\times \dot B^{\frac{3}{p}-1}_{p,1}$ 
 with  $1\leqslant p<3$ (see the paper by H. Abidi and M. Paicu \cite{AP}, 
later extended to the whole range $1\leqslant p<6$ in \cite{dan8});
 \item the well-posedness fo $u_0\in \dot B^{\frac{3}{p}-1}_{p,r}$ with 
 arbitrarily large $p$ and $r$ if the density is at least H\"older continuous, see the work by B. Haspot in \cite{Haspot}.
 \end{itemize}
 To our knowledge, the question of the solvability of \eqref{inhom} for $u_0$ only in $\BMO$ and $\rho_0$ with no regularity 
 has not yet been addressed. 
In the present paper, we will take a step toward resolving this issue by considering $u_0$ in new critical regularity spaces 
whose definition is inspired by that of $\BMO.$  
As for the initial density $\rho_0,$ it is not required to be smooth: it just
has to be sufficiently close to $1$ in the uniform metric. 
\medbreak

%%%%%%%%%%%%%%%%%%%%%%%%%%%%%%%%%%%%%%%%%%%%%%%%%%%%%%%

\section{The main result}

Before stating our main result, introducing the functional setting is in order.
Let us first specify the space $U_\alpha$ for initial data. 
Let $\alpha\in(0,1)$ and let $\Psi$ be a nonzero function in the Schwartz space $\cS(\R^n)$ with zero average and satisfying 
$|\widehat\Psi(\xi)|>0$ whenever 
$1/2<|\xi|<2$. Put $\Psi_t(x):=t^{-n}\Psi(x/t)$. We set for any tempered distribution $f,$
$$\|f\|_{ U_{\alpha}}:= \underset{\atop{r>0}{x_0 \in \mathbb R^n}}{\sup} \, \Bigl( \frac{1}{r^{n-2 \alpha - 2 }} \int_0 ^{r} \int_{B(x_0, r)} 
|(\Psi_t\ast f)(y)|^2 \, t^{-1-2\alpha} \, dy\,dt \Bigr)^{\frac {1}{2}}.$$
A key example of functions $\Psi_t$ is  $\Psi_t:=t \partial_t  \Phi_{t^2}$
where  $\Phi_t$ stands for the heat kernel in $\mathbb R^n$,  namely
\begin{equation}\label{eq:Phi}
\Phi_t (x) = \frac{e^{{\frac{-|x|^2}{4t}}}}{(4\pi t)^{\frac{n}{2}}}\cdotp\end{equation}
It will be shown in Appendix (see Proposition \ref{p:density}) that if, in addition, $\alpha\leq -1+ n/2$ then, any Schwartz function $f$
satisfies $\|f\|_{U_\alpha}<\infty.$
This motivates the following definition of $U_\alpha.$
\begin{defin}
\label{tratatatatata}
Assume that $\alpha\in(0,-1+n/2]\cap(0,1)$.
%(or that $n\geq4$ and $\alpha\in(0,1)$). 
Then, we define $U_\alpha$ to be the completion of the Schwartz  space $\cS(\R^n)$ 
for the norm $\|\cdot\|_{U_\alpha}.$
%Let $\alpha\in (0,1)$ and $\Psi$ be a nonzero Schwartz function  with zero average. We say that a tempered distribution $f$ belongs to $U_{\alpha}(\mathbb R^n)$ if 
%$$\|f\|_{ U_{\alpha}}:= \underset{\atop{r>0}{x_0 \in \mathbb R^n}}{\sup} \, \Bigl( \frac{1}{r^{n-2 \alpha - 2 }} %\int_0 ^{r} \int_{B(x_0, r)} 
%|(\Psi_t\ast f)(y)|^2 \, t^{-1-2\alpha} \, dy\,dt \Bigr)^{\frac {1}{2}} < \infty,$$
%where $\Psi_t(x):=t^{-n} \Psi(x/t)$ for all $t>0.$	
%\red{I hope that this definition rules out constant functions.}
\end{defin}
 By Theorem 3.2 of Liang, Sawano, Ullrich, Yang and Yuan~\cite{liang2012new}, with
$$
s=\alpha,\qquad p=q=2,\qquad
\tau=\frac{n-2\alpha-2}{2n},
$$
on Schwartz functions the corresponding continuous local-means norm is an equivalent norm of their Morrey--Triebel--Lizorkin space $\dot F^{\alpha,\tau}_{2,2}(\mathbb R^n)$
and, in particular, is independent, up to equivalence, of the particular admissible function $\Psi$.  Since balls and cubes of comparable radii give equivalent norms, this result applies to the definition of $U_\alpha$ above. 

In view of the aforementioned theorem, we identify $U_\alpha$ with the closure of $\mathcal S(\mathbb R^n)$ in the homogeneous Triebel–Lizorkin-type space $\dot F^{\alpha,(n-2-2\alpha)/(2n)}_{2,2}(\mathbb R^n),$
which is naturally realized as a space of distributions modulo polynomials in $\mathcal S'(\mathbb R^n)/\mathcal P$. Whenever a concrete representative is needed, we choose the normalized representative that satisfies
$e^{t\Delta}f\to 0$ in $\mathcal S'(\mathbb R^n)$ as $t\to\infty.$
\medbreak
Observe that since, in the above definition, integration with respect to the $t$ variable is restricted to $0\leqslant t\leqslant r,$ we have the continuous embedding $U_\alpha\hookrightarrow U_{\alpha'}$
whenever $\alpha'\leqslant\alpha,$ and the space $U_{-1}$  looks like ${\rm BMO}^{-1}.$
\medbreak
Our space $U_\alpha$ is scaling invariant in the sense of \eqref{eq:critical0} (see
Proposition \ref{prop1} in the Appendix).  It also resembles  the space $Q^\alpha_{\frac{n}{\alpha+1},2}(\mathbb R^n)$ of the family $Q^\alpha_{p,q}(\mathbb R^n)$ (first introduced in~\cite{sibir}) defined as follows: a measurable function
$f:\R^n\to\R$ is in $Q^\alpha_{p,q}(\mathbb R^n)$ if
\begin{equation}\label{eq:Q}\sup_I\frac{1}{|I|^{1-\frac{q}{p}}}\int_I\int_I \frac{|f(x)-f(y)|^q}{|x-y|^{n+q\alpha}}\, dx dy<\infty,
\end{equation}
where the supremum is taken over all cubes $I$ in $\mathbb R^n$. See~\cite{hernya} for the properties of those spaces and references 
therein for previous important works on them. %Consequently, despite appearances,  the space  $U_\alpha$  does not depend on the specific choice of $\Psi$. This as well as the proof of the fact that the spaces $U_\alpha(\mathbb R^n)$ and $Q^\alpha_{\frac{n}{\alpha+1},2}(\mathbb R^n)$ coincide is a  consequence of the Corollary 4.5 in~\cite{hernya}.
Our spaces $U_\alpha$ are also closely related to those
that have been used by J. Xiao in~\cite{Xiao} to study the classical  Navier--Stokes equations.
Their use in the density dependent case  seems to be totally new, though.
\medbreak
Before giving the definition of the solution space for the velocity field, introducing several
families of \emph{tent spaces} is in order. 
%In the rest of the section, it will be convenient to  denoted by  $ \R^n\times\R_+$
%the set of pairs $(x,t)$ such that $t>0$ and $x\in\R^n.$
Let us first recall the definition of the classical parabolic tent spaces.
\begin{defin}\label{palatprvo}
 Let $p\in[1,\infty).$ A measurable function $u: \R^n\times\R_+\rightarrow \mathbb R$ belongs to the $(\infty,p)$-parabolic tent space $\mathbb T^{\infty,p}( \R^n\times\R_+)$ if
$$\|u\|_{\mathbb T^{\infty,p}( \R^n\times\R_+)}:=
\underset{\atop{r>0}{x_0 \in \mathbb R^n}}{\sup} \biggl(\frac{1}{r^n}\int_0^{r^2}\!\!\int_{B(x_0,r)}|u(y,t)|^p\, dy\, dt\biggr)^{\frac1p}<\infty.$$
\end{defin}
Tent spaces, along with convex integration (see e.g.~\cite{buck}) and dual evolution techniques (see e.g.~\cite{kisnaz} and~\cite{vasvin}) are examples of important tools, originating outside PDE, that proved to be useful to tackle challenging problems related to the Navier-Stokes equations.

 To our knowledge, \emph{elliptic} tent spaces were first introduced in the article~\cite{cms}. They have been studied
 in the subsequent papers \cite{aus-1,aus,aus1,aus2}. The parabolic counterpart of these spaces  (see the above definition) was  considered and successfully applied
  by P. Auscher and D. Frey in~\cite{aus}   to solve the homogeneous Navier--Stokes equations \eqref{ns} with initial data in $\BMO$ 
(see also~\cite{notes}).   Compared to the original article~\cite{kochtataru} by H. Koch and D. Tataru, 
using tent spaces turned out to allow one to handle more general parabolic systems with, possibly, 
rough coefficients. For more on tent spaces, see the recent papers~\cite{ben1} and~\cite{ben2}.
\medbreak
In order to state our main result, we need
another two families of tent spaces.
The first (resp. second) one is the space in which the first (resp. second) order space derivatives of the 
velocity are expected
to be, if one starts with $u_0$ in $U_\alpha.$
 \begin{defin}
 \label{prostranstoreshenij}
 Let $\beta \in \mathbb R$. 
   We denote by  $\mathbf{T}^{\infty, 2} (t^\beta dy dt)$
 the space of functions $u$ in $L^2_{loc}(\R^{n}\times\R_+;\R)$
such that  
$$\|u\|_{\mathbf{T}^{\infty, 2} (t^\beta dy dt)}:=\underset{\atop{r>0}{x_0 \in \mathbb R^n}}{\sup} \biggl(\frac{1}{r^{n+2\beta-2}} \int_0^{r^2} \!\!\int_{B(x_0,\,r)} \big| u(y,t) \big|^2 t^\beta dy \, dt \biggr)^\frac{1}{2} < \infty.$$
 We denote by  $T^{\infty, 2} (t^\beta dy dt)$
 the space of functions $u$ in $L^2_{loc}( \R^n\times\R_+;\R)$
such that  
$$\|u\|_{T^{\infty, 2} (t^\beta dy dt)}:=\underset{\atop{r>0}{x_0 \in \mathbb R^n}}{\sup} \biggl(\frac{1}{r^{n+2\beta-4}} \int_0^{r^2} \!\!\int_{B(x_0,\,r)} \big| u(y,t) \big|^2 t^\beta dy \, dt \biggr)^\frac{1}{2} < \infty.$$
 \end{defin}
 The same notation  will be used for \emph{vector valued} functions.
   \medbreak
To finish with, we set for  any measurable function $z: \R^n\times\R_+\to\R$, 
$$\|z\|_X:=\underset{\atop{r>0,\, x_0 \in \mathbb R^n}{0<t<r^2}}{\sup} r^{1-\frac n2}
\|z(t)\|_{L^2(B(x_0,r))}$$ and denote by $X$ the corresponding function space.
\medbreak
One can  now introduce the solution space for the velocity.
\begin{defin}\label{def:E}
 Let $\alpha \in \mathbb R$. We  denote by 
 $E_\alpha$ the set of time-dependent solenoidal vector fields $u$ with components in $X$ such that, in addition, 
$$\displaylines{ u\in L^\infty(\R_+;\dot B^{-1}_{\infty,\infty}),\quad
\partial_t u,\nabla^2 u\in T^{\infty,2}(t^{1-\alpha}dy dt)\andf \nabla u\in \mathbf{T}^{\infty,2}(t^{-\alpha}dy dt).
%,\crt^{\frac{1}{2}}u\in L^\infty(\mathbb R^{n+1}_+)\andf u\in L^{\infty}(\mathbb %R_+;\dot B_{\infty,\infty}^{-1}).
}$$
 \end{defin}
It is now time to state the main result of this paper:
\begin{theorem}
\label{key} Let $\alpha\in(0,1/2]$.
 There exists a positive number $\varepsilon_0$ such that  if 
the data $(\rho_0,u_0)$ (with divergence-free $u_0$)
satisfy  
 \begin{equation}\label{eq:smalldata}
 \|u_0\|_{U_\alpha}+\|\rho_0-1\|_{L^\infty(\R^{3})}\leqslant \varepsilon_0,\end{equation}
  then there exist global-in-time distributional solutions $(\rho,u,P)$ to the system~\eqref{inhom} with 
$u$ in the space $E_\alpha$ and  $\rho$ such that
    $$\|\rho-1\|_{L^\infty( \R^n\times\R_+)}  \leqslant \|\rho_0-1\|_{L^\infty(\R^{n})}.$$   \end{theorem}
  % \begin{rem} Although the formal estimates leading to this statement are independent of
  %the dimension of the ambient space, we here focus on the physical cases $n=2$ and $n=3$ for the sake of simplicity.
  %This will be helpful for constructing the solutions.\end{rem}
\begin{rem} Having $u$ in $E_\alpha$  does not guarantee  
  that $\nabla u$ belongs to $L^1_{loc}(\R_+;L^\infty).$ 
  Due to the partially hyperbolic nature of \eqref{inhom}, this makes the uniqueness 
  issue rather challenging if  the density has no smoothness 
  (see  the work by B. Haspot in \cite{Haspot} and the recent preprint \cite{HNSWZZ}
  for the case where $\rho$ has some H\"older regularity). 
    %If $n=2,$ the recent papers \cite{HSWZ,Sk} guarantee that it holds true. 
%  A maybe more affordable question is that of the uniqueness in the case where the initial density
 % has some smoothness (see the work by B. Haspot in \cite{Haspot} for a related question
  %in a different functional framework).
  \end{rem}
  
  Our choice of the spaces $U_\alpha$ is motivated by their similarity with the space ${\rm BMO}^{-1}$
  defined in \eqref{eq:BMO-1}. They are connected with  the family of $Q_\alpha^\beta(\R^n)$ spaces that have been first introduced 
in \cite{EJPX} and used to investigate the well or ill-posedness
issues of the homogeneous incompressible Navier--Stokes equations with fractional Laplacian  in \cite{LZ} and \cite{WX}. 

Compared to the constant density case, the difficulty here is that the system under consideration
is not fully parabolic. This led us to consider a new scale of tent spaces.
\medbreak
We conclude this section with a brief explanation of the proof. Essentially, we have to control the velocity field in the space $E_\alpha$ in terms of $\|u_0\|_{U_\alpha}$ in the case it is  small enough  in $U_\alpha$ and $\rho_0$ is close  to $1.$ 
Since the density control is trivial, this consists in  showing that:
\begin{enumerate}
\item[1.]  the free solution of the non-stationary Stokes system associated with an initial velocity in $U_\alpha$ indeed belongs to $E_\alpha$; 
\item[2.]  the nonlinear terms in the velocity equation can be controlled in the tent space  $T^{\infty,2}(t^{1-\alpha}dy dt)$; 
\item[3.]  the solution of the Stokes system with a source term in $ T^{\infty,2}(t^{1-\alpha}dy dt)$ is  in $E_\alpha$.
\end{enumerate}
We shall establish beforehand the continuity of the Leray projector onto $ T^{\infty,2}(t^{1-\alpha}dy dt)$. This  will allow us to consider the heat equation rather than the Stokes system in the above steps. Verifying all these points will be the subject of the following section. 
In section \ref{s:existence}, we will rigorously prove the existence of a global solution satisfying the properties of the main theorem. 
%\medbreak
%The rest of the paper unfolds as follows. 
%In the next section, we introduce all the functional setting pertaining to tent spaces 
%and establish a few key properties. The third section is devoted to the proof of the fact that the maximal regularity operator is bounded on our tent spaces and to the boundedness of the Leray projector on these spaces. In the fourth section we prove our existence result. 
In appendix,  we demonstrate some auxiliary results which have been of importance in our paper.
%like  a density result for $U_\alpha$ and  its scaling invariance.
%\medbreak
%\blue{ We finally mention one open question related to Theorem~\ref{key}: does a similar result hold in more general fluid domains~? The authors plan to attack this question in the near future. To this end, it seems natural for us first to establish a general description of the spaces $U_\alpha$, for instance, as it was done in~\cite{tsvas} for the scale of Triebel--Lizorkin spaces.}
\medbreak
Throughout this paper, the signs $\lesssim$ and $\gtrsim$ indicate that the left-hand (right-hand) side of an inequality is less than the right-hand (left-hand) side multiplied by a `harmless' constant. The sign $\asymp$ means that we have both bounds $\lesssim$ and $\gtrsim$ simultaneously.

%%%%%%%%%%%%%%%%%%%%%%%%%%%%%%%%%%%%%%%%%%%%%%%%%%%%%%

\section{The key inequalities}

This section is devoted to proving the key inequalities that lead to our main theorem. 
It is divided into four parts. In the first three parts, we establish linear estimates. 
In the last one,  we  handle the convection term in the velocity equation of \eqref{inhom}.

\subsection{The Riesz operators in tent spaces}

The following result  ensures that the Leray projector maps the space
$T ^{\infty,\, 2} \bigl( t^{\beta}dy dt \bigr)$ to itself in any dimension $n.$  It  will enable us 
to deduce the boundedness of $\nabla^2u$ in this space from that of $\Delta u,$ and 
to reduce the proof of estimates for the evolutionary Stokes system 
to the corresponding ones for the heat equation. 
\begin{lemma}\label{lempiat} 
 Let   $\beta <2.$  In any dimension $n,$ the Riesz transforms $R_1,\dots,R_n$ are bounded in the tent space $T ^{\infty, 2} \bigl( t^{\beta}  dy  dt \bigr)\cdotp$
\end{lemma} \begin{proof}
It suffices to prove the result for $R_1.$  Consider a function $u \in T^{\infty, 2} \bigl( t^\beta dy dt \bigr).$ Take a point $x_0 \in \mathbb R^n$ and $r>0.$ What we need to estimate is
\begin{align*}
I^2 &:= \int_0^{r^2}\!\!\int_{B(x_0,\, r)} \big| R_1 u(y,t) \big|^2 t^\beta \, dy \, dt  \\
&=\int_0^{r^2} \!\!\int_{B(x_0,\, r)} \Big| \sum_{i \geqslant 0} \int_{F_i} u(x,t) \frac{y_1 - x_1}{\left| y - x\right|^{n+1}} dx  \Big|^2 t^\beta \, dy \, dt\\  
&=\int_0^{r^2}\!\! \int_{B(x_0,\, r)} \Big| \sum_{i \geqslant 0} \varphi_i (y,t) \Big|^2 \, t^\beta \, dy dt
\quad\hbox{with }\  \varphi_i(x,t):=\int_{F_i} u(x,t) \frac{y_1 - x_1}{\left| y - x\right|^{n+1}} \,dx,
\end{align*}
where we have denoted $F_0:=B(x_0,2r)$ and $F_i := B(x_0, 2^{i+1} r) \backslash B(x_0, 2^i r)$ for $i \geqslant 1.$ 
On top of that, the Minkowski inequality yields
$$I \leqslant \sum_{i \geqslant 0} \biggl( \int_0^{r^2}\!\!\int_{B(x_0,\, r)} \left|  \varphi_i (y,t) \right|^2 \, t^\beta \, dy \, dt \biggr)^\frac{1}{2}\cdotp$$
On the one hand, since $R_1$ is a bounded operator on $L^2(\R^n),$
we have
\begin{align*}
\int_0^{r^2} \!\!\int_{B(x_0,\, r)} \left|  \varphi_0 (y,t) \right|^2 \, t^\beta \, dy \, dt  &=  \int_0^{r^2} \| R_1 \bigl( \mathbbm{1}_{B(x_0, \, 2r)} \, u(\cdot , t) \bigr) \|_{L^2(B(x_0, \, r))}^2\, t^\beta  \, dt  \\
&\leqslant \int_0^{r^2} \| R_1 \bigl( \mathbbm{1}_{B(x_0, \, 2r)} \, u(\cdot , t) \bigr) \|_{L^2 (\mathbb R^n)}^2\, t^\beta  \, dt\\ 
&\lesssim \int_0^{r^2}\!\! \int_{B(x_0,\,2r)}\left|u(y,t) \right|^2 t^\beta \, dy \, dt.
\end{align*}

On the other hand, if $i \geqslant 1,$ then we use 
the fact that $|y-x|\geq 2^{i-1}r$ for all $x\in F_i$ and $y\in B(x_0,r)$
and the Cauchy--Schwarz inequality in the integral with respect to $x$ to infer:
\begin{align*}
\int_0^{r^2}\!\! \int_{B(x_0,\,r)}\left|\varphi_i (y,t) \right|^2 t^\beta \, dy \, dt 
&= \int_0^{r^2}\!\! \int_{B(x_0,\,r)}\left|\int_{F_i} u(x,t) \frac{y_1 - x_1}{\left|y-x \right|^{n+1}} \,dx \right|^2 t^\beta \, dy \, dt \\
&\lesssim\int_0^{r^2}\!\! \int_{B(x_0,\,r)} \bigl(2^i r\bigr)^{-n} 
\int_{F_i} |u(x,t)|^2  \,dx \  t^\beta dy \, dt  \\
&\lesssim  \int_0^{r^2}\!\! \int_{B(x_0,\,2^{i+1} r)} 2^{-in} \left| u(x,t)\right|^2 t^\beta \,dx\, dt.
\end{align*}
Putting  the estimates above together yields  
\begin{align*}
I &\leqslant \sum_{i \geqslant 0} \Biggl( \int_0^{r^2} \!\!\int_{B(x_0,\,2^i r)} 2^{-in} \left| u(x,t)\right|^2 t^\beta dx\, dt \Biggr)^\frac{1}{2} \\
&\lesssim \sum_{i \geqslant 0} 2^{-\frac{in}{2}} \bigl(2^i r\bigr)^{\frac{n}{2} +\beta -2} \Biggl( \frac{1}{{(2^i r)}^{n+2\beta -4}} \int_0^{{(2^i r)}^2} \!\!\int_{B(x_0,\,2^i r)}  \left| u(x,t)\right|^2 t^\beta dx\, dt \Biggr)^\frac{1}{2} \\ &\lesssim r^{\frac{n}{2}+\beta-2} \cdot \| u \|_{T^{\infty, 2}(t^\beta dy dt)} \cdot \sum_{i \geqslant 0} 2^{i(\beta -2)} \\
&\lesssim r^{\frac{n}{2}+\beta-2} \| u \|_{T^{\infty, 2}(t^\beta dy dt)},
\end{align*}
and the lemma follows.
 \end{proof}

%%%%%%%%%%%%%%%%%%%%%%%%%%%%%%%

\subsection{Estimates for the free solution of the heat equation}

%Let $u_L$ designate  the solution of the free heat equation in $\R^n\times\R_+$ with initial data $u_0$ in $U_\alpha.$ 
The goal of this part is to establish the following result. 
\begin{lemma}\label{lemodin}   
Assume that  $u_0  \in U_\alpha$ for some $\alpha \in (0, \, 1).$ Then, $u_L:=e^{t\Delta}u_0$ satisfies 
\begin{multline}\label{eq:uL}
\sup_{t>0} \|u_L(t)\|_{\dot B^{-1}_{\infty,\infty}} +\|(\partial_t u_L,\nabla^2 u_L)\|_{T^{\infty,2} (t^{1-\alpha} dy\, dt)}\\+\|u_L\|_X+\|\nabla u_L\|_{\mathbf{T}^{\infty, 2} (t^{-\alpha} dy dt)}\lesssim \|u_0\|_{U_\alpha}.\end{multline}
Furthermore, we have 
\begin{equation}\label{eq:uLbis}
\sup_{t>0} t^{\frac12}\|u_L(t)\|_{L^\infty}\lesssim \|u_0\|_{U_\alpha}.\end{equation}
\end{lemma}
\begin{proof}
Inequality  \eqref{eq:uLbis} is only a matter of scaling: since 
$U_\alpha$ satisfies \eqref{eq:critical} (see Prop. \ref{prop1} in the Appendix), 
 and is invariant by translations, it is continuously embedded in $\dot B_{\infty, \infty}^{-1}$ (see \cite[Prop. 5.31]{BCD}). 
 Recall that the norm in $\dot B^{-1}_{\infty,\infty}$ can be  characterized  in terms of the heat flow  as follows (see \cite[Thm. 2.34]{BCD}):
 \begin{equation}\label{eq:char}
 \sup_{t>0} t^{\frac\sigma2}\| e^{t\Delta} a\|_{L^\infty(\R^n)}\simeq \|a\|_{\dot B^{-\sigma}_{\infty,\infty}},\quad\sigma>0. 
 \end{equation}
 This gives \eqref{eq:uLbis}. Similarly, due to \eqref{eq:char}, we have
\begin{align*}
\sup_{t>0}\|u_L(t)\|_{\dot B^{-1}_{\infty,\infty}}&=
\sup_{s,t>0} s^{\frac12}  \|e^{(t+s)\Delta}u_0\|_{L^\infty}\\
&\leq \sup_{s,t>0} (s+t)^{\frac12}  \|e^{(t+s)\Delta}u_0\|_{L^\infty}
=\|u_0\|_{\dot B^{-1}_{\infty,\infty}}\lesssim \|u_0\|_{U_\alpha}.\end{align*}
% Next, to bound the first term of \eqref{eq:uL}, recall that $U_\alpha$ is continuously embedded in $\BMO.$  Since, according to~\cite{germ}, the norm $\BMO$ is propagated by heat flow, we find that $u_L$ is in  $L^{\infty}(\R_+;\BMO)$  whence $u_L\in L^\infty(\mathbb R_+; \dot B_{\infty,\infty}^{-1}).$
  To bound the second term of Inequality \eqref{eq:uL},
 fix $x_0$ in $\mathbb R^n$ and $r>0.$  Using the short notation $B_r:=B(x_0,r),$
and observing that $u_L(t)=\Phi_t\ast u_0,$ we get
  \begin{align*}
J:&= \frac{1}{r^{n-2\alpha-2}} \int_0^{r^2}\!\! \int_{B_r} \big| \partial_t  \bigl(\Phi_t \ast u_0 \bigr) (y) \big|^2 \, t^{1-\alpha} dy\,dt\\
&=\frac{1}{r^{n-2\alpha-2}} \int_0^{r^2} \!\!\int_{B_r} \big| (\partial_t  \Phi_t) \ast u_0 (y) \big|^2 \, t^{1-\alpha} dy\,dt \\
&=\frac{1}{r^{n-2\alpha-2}} \int_0^r\!\! \int_{B_r} 
\big| (s\partial_s  \Phi_{s^2}) \ast u_0 (y) \big|^2 \, s^{-1-2\alpha} dy\,ds
\leqslant \|u_0\|_{U_\alpha}^2,
%&=:\frac{1}{r^{n-2-2\alpha}} \int_0^{r^2} \!\!\int_{B(x_0, \, r)} \big| \Phi_{t} \ast g_0 (y) \big|^2 \, t^{1-\alpha} dy\,dt 
\end{align*}
where the second equality above follows from a change of variables. Hence
we are done, thanks to Theorem 3.2 in~\cite{liang2012new} which asserts that the space  $U_\alpha$ does not depend on the specific choice of the mollifier $\Psi$ in Definition~\ref{tratatatatata} (here we choose
$\Psi_s:=s\partial_s\Phi_{s^2}$).
Because $\Delta u_L=\partial_t u_L$ and $\partial_i\partial_j u_L =-R_i R_j \Delta u_L$
with $R_i$ being the Riesz transform with respect to the variable $x_i,$
Lemma \ref{lempiat} guarantees that $\nabla^2 u_L$ also satisfies the desired estimate. 
\medbreak
Next, in order to bound $\|u_L\|_X,$  we note that, for any $0<t<r^2,$ we have
\begin{equation}\label{eq:uL1}u_L(t)=u_L(r^2)-\int_t^{r^2} \partial_s u_L(s)\,ds.\end{equation}
Hence, we have for all $x_0\in\R^n$,
$$\|u_L(t)\|_{L^2(B(x_0,r))}\leqslant \|u_L(r^2)\|_{L^2(B(x_0,r))}
+\int_t^{r^2}\|\partial_s u_L\|_{L^2(B(x_0,r))}\,ds.$$
From H\"older Inequality and Inequality \eqref{eq:uLbis}, we find that
\begin{equation}\label{eq:uL2}\|u_L(r^2)\|_{L^2(B(x_0,r))}\lesssim r^{\frac n2}\|u_L(r^2)\|_{L^\infty}\leqslant  r^{\frac n2-1} \|u_0\|_{U_\alpha}.
\end{equation}
In light of the Cauchy-Schwarz inequality and of the bound on $\partial_tu_L$
given by \eqref{eq:uL}, we have for $\alpha>0,$
\begin{align*}\int_t^{r^2}\|\partial_s u_L\|_{L^2(B(x_0,r))}\,ds&\leqslant \biggl(\int_0^{r^2}\!\!\int_{B(x_0,r)}|\partial_s u_L|^2
s^{1-\alpha}\,dy ds\biggr)^{\frac12}\biggl(\int_0^{r^2}s^{\alpha-1}\,ds\biggr)^{\frac12}\\
&\lesssim r^{\frac n2-1} \|u_0\|_{U_\alpha},\end{align*}
which, together with \eqref{eq:uL1} and \eqref{eq:uL2}, and the definition of $\|\cdot\|_X$  implies that
$$\|u_L\|_X\lesssim \|u_0\|_{U_\alpha}.$$ 
Let us finally bound $\nabla u_L.$ 
 Let $\partial_k$ denote $\partial/\partial x_k$ for $k$ from $1$ to $n .$ Recall that $\Phi_t$ stands for the heat kernel in $\mathbb R^n$ and write as follows:
\begin{equation*}
\int_0^{r^2} \!\!\!\int_{B_r} \big| \nabla (\Phi_t \ast  u_0) (y)\big|^2 \, t^{-\alpha} \, dy dt 
= \sum_{k=1}^n \int_0^r\!\! \int_{B_r} \big| (\tau \partial_k \Phi_{\tau^2}) \ast u_0 (y)\big|^2 \, \tau^{-1-2\alpha} \, dy d\tau.
\end{equation*}
Note that all  the functions 
$$\eta_k (\tau, \, x):= \tau \, \partial_k \Phi_{\tau^2} =\tau\, \partial_k \Bigl(\pi^{-n} \, \tau^{-n} \, e^{- \frac{|x|^2}{\tau^2}}\Bigr) = \tau^{-n} \pi^{-n} e^{- \frac{|x|^2}{\tau^2}} \Bigl(\frac{-2x_k}{\tau}\Bigr)$$ 
satisfy $\eta_k (\tau, x) = \tau^{-n}\, \eta_k (x/\tau)$ where $\eta_k (z) := \pi^{-n} e^{-|z|^2} (-2 z_k),$ have a mean value $0$ and are such that $|\eta_k (z)| \lesssim (1+|z|)^{-M}$
for all $M\in\N.$ 
Consequently, one can take advantage of  Theorem 3.2 in~\cite{liang2012new}
and conclude  that
\begin{align*}
\int_0^{r^2} \!\!\!\int_{B_r} \left| \nabla \bigl( \Phi_t \ast u_0\bigr) (y) \right|^2 t^{-\alpha }\, dy\,dt &= \sum_{k=1}^n \int_0^r\!\! \int_{B_r} \left| \eta_k \ast u_0(y,t) \right|^2 t^{-1-2\alpha} dy \, dt \\ 
&\lesssim r^{n-2\alpha-2} \|u_0\|_{U_\alpha}^2.
\end{align*}
This completes the proof of the lemma. 
\end{proof}

%%%%%%%%%%%%%%%%%%%%%%%%%%%%%%%%%%%%%%%%%%%%%%%%%

\subsection{Estimating the Duhamel term}

Let $v$ be the solution of the heat equation with null initial data
and source term $f,$  namely
\begin{equation}\label{eq:v}
v(t):=\int_0^te^{(t-s)\Delta} f(s)\,ds,\qquad t\in\R_+. 
\end{equation}
We want to prove  that $v$ belongs to  $E_\alpha$
whenever $f$ is in the tent space  $T^{\infty, 2} (t^{1-\alpha} dy dt)$ for some $\alpha\in(0,1).$  
\smallbreak
Let us first  show the continuity on  $T^{\infty, 2} (t^{1-\alpha} dy dt)$ 
of the following maximal regularity operator $M_+$ acting on functions 
$f: \R^n\times\R_+\to\R$:
 $$M_+ f(y,t) := \int_0^t \Delta e^{(t-s) \Delta} f(y,s)\, ds.$$ 
\begin{lemma}
\label{lemdva} For any $\beta<1$ and dimension $n\geqslant 1,$ we have 
$$M_+:  T^{\infty, 2} (t^\beta dy dt) \rightarrow  T^{\infty, 2} (t^\beta dy dt).$$ 
\end{lemma}
\begin{proof}
It is inspired by that of Theorem 3.2 of~\cite{aus2} (see also~\cite{notes}).
 Fix some $f \in T^{\infty, 2}(t^\beta dy\, dt),$ 
then  $x_0 \in \mathbb R^n$ and $r>0.$ Let $B_r:=B(x_0,r).$ We want to estimate 
$$I := \biggl(\int_{B_r} \int_0^{r^2} \big| M_+ f(y,t) \big|^2 t^\beta dt \, dy\biggr)^{1/2}\cdotp$$
For any $j\geqslant 0,$ denote
$$I_j := \biggl(\int_{B_r} \int_0^{r^2} \big| M_+ f_j (y,t) 
\big|^2 t^\beta dt \, dy\biggr)^{1/2},$$
where
$$f_0(x,t) := f(x,t) \cdot \mathbbm{1}_{B(x_0, \, 2r)}(x)\cdot \mathbbm{1}_{[0, \, r^2]}(t)$$ 
and if $j \geqslant 1,$ 
$$f_j(x,t):= f(x,t)\cdot \mathbbm{1}_{B(x_0, \, 2^{j+1}r) \backslash B(x_0,\, 2^j r)}(x) \cdot \mathbbm{1}_{[0, \, r^2]}(t).$$ 
Owing to the Minkowski inequality, we have $$I \leqslant \sum\limits_{j \geqslant 0}  I_j.$$ 
To handle $I_0,$ we deduce from the fact that $M_+$ is bounded  on $L^2 (\mathbb R^{n}\times\R_+, t^\beta dt \, dy)$  when $\beta <1$ (by the de Simon theorem in \cite{DeSimon}) that 
 \begin{align*}
I_0^2 &= \int_{B_r}\int_0^{r^2} \big|M_+f_0(y,t) \big|^2 t^\beta dy\,dy \\   
&\leqslant\int_{\mathbb R^n} \int_0^{\infty} \big| M_+ f_0 (y,t)\big|^2 t^\beta dt\,dy \lesssim \int_{B_r} \int_0^{r^2} \big| f (y,t)\big|^2 t^\beta dt\,dy.
\end{align*}
Next, for all $j\geqslant 1,$ we have
\[M_+f_j(t)=
\int_{t/2}^t \Delta e^{(t-s)\Delta}f_j(s)\,ds+
\sum_{k\geqslant1}\int_{2^{-k-1}t}^{2^{-k}t}
\Delta e^{(t-s)\Delta}f_j(s)\, ds.\]
Hence, by Minkowski's inequality,
\[I_j\leqslant I_{j,0}+\sum_{k\geqslant 1}I_{j,k},\]
where
\[I_{j,0}:=\left(
\int_0^{r^2}\left\|\int_{t/2}^t\Delta e^{(t-s)\Delta}f_j(s) ds
\right\|_{L^2(B_r)}^2t^\beta dt\right)^{1/2}\]
and, for $k\geqslant 1$,
\[
I_{j,k}
:=
\left(
\int_0^{r^2}
\left\|
\int_{2^{-k-1}t}^{2^{-k}t}
\Delta e^{(t-s)\Delta}f_j(s) ds
\right\|_{L^2(B_r)}^2
t^\beta dt
\right)^{1/2}.
\]
To estimate the near-diagonal term $I_{j,0},$  we leverage 
the following off-diagonal estimate that holds true for all $N>0$:
\[\left\|
\Delta e^{(t-s)\Delta}f_j(s)\right\|_{L^2(B_r)}
\lesssim K_j^N(t-s)
\|f_j(s)\|_{L^2}\with
K_j^N(u):=u^{-1}\left(1+\frac{(2^jr)^2}{u}
\right)^{-N}.\]
Clearly, for all $0<t\leqslant r^2$, we have
\[\int_0^{t/2}K_j^N(u)du
\leqslant
(2^jr)^{-2N}
\int_0^{t/2}u^{N-1}du
\lesssim
(2^jr)^{-2N}t^N
\lesssim 2^{-2jN}.
\]
Hence using convolution inequality and the fact that $t^\beta\simeq s^\beta$
for $s\in[t/2,t]$ gives
%The same estimate holds after interchanging the $s$ and $t$ variables. Therefore, Schur's lemma, applied in the weighted space
%$L^2((0,r^2),t^\beta dt)$, gives
\[I_{j,0}
\lesssim
2^{-2jN}
\left(
\int_0^{r^2}
\|f_j(s)\|_{L^2}^2s^\beta ds
\right)^{1/2}.
\]
We now estimate $I_{j,k}$ for $k\geqslant 1$. By the Cauchy--Schwarz inequality with respect to $s$,
\[
\begin{aligned}
I_{j,k}^2
&\leqslant
\int_0^{r^2}
\int_{2^{-k-1}t}^{2^{-k}t}
\left\|
(t-s)\Delta e^{(t-s)\Delta}f_j(s)
\right\|_{L^2(B_r)}^2
ds \;2^{-k}t^{\beta-1}\, dt.
\end{aligned}
\]
We thus obtain
\[
\begin{aligned}
I_{j,k}^2
&\lesssim
\int_0^{r^2}
\int_{2^{-k-1}t}^{2^{-k}t}
2^{-k}t^{\beta-1}\left(1+\frac{(2^jr)^2}{t-s}
\right)^{-2N}
\|f_j(s)\|_{L^2}^2\, ds dt.
\end{aligned}
\]
For $k\geqslant1$ and
\[
2^{-k-1}t\leqslant s\leqslant2^{-k}t,
\]
we have
\[
t-s\simeq t.
\]
Changing the order of integration therefore gives
\[
\begin{aligned}
I_{j,k}^2
&\lesssim
\int_0^{2^{-k}r^2}
\int_{2^ks}^{2^{k+1}s}
2^{-k}t^{\beta-1}
\left(
1+\frac{(2^jr)^2}{t-s}
\right)^{-2N}
\|f_j(s)\|_{L^2}^2
 \,dt ds \\
&\lesssim
2^{-4jN}r^{-4N}
\int_0^{2^{-k}r^2}
2^{-k}\|f_j(s)\|_{L^2}^2
\left(
\int_{2^ks}^{2^{k+1}s}
t^{\beta+2N-1} dt
\right)ds.
\end{aligned}
\]
Consequently,
\[
\begin{aligned}
I_{j,k}^2
&\lesssim
2^{-4jN}
\int_0^{2^{-k}r^2}
2^{-k}(2^ks)^{\beta+2N}
r^{-4N}
\|f_j(s)\|_{L^2}^2 ds \\
&\lesssim
2^{-k(1-\beta)}2^{-4jN}
\int_0^{r^2}
\|f_j(s)\|_{L^2}^2s^\beta ds,
\end{aligned}
\]
where we used $2^ks\leqslant r^2$. Hence
\[
I_{j,k}
\lesssim
2^{-k(1-\beta)/2}2^{-2jN}
\left(
\int_0^{r^2}
\|f_j(s)\|_{L^2}^2s^\beta ds
\right)^{1/2}.
\]
Since $\beta<1$, the series with respect to $k$ is convergent. Combining the estimates for $I_{j,0}$ and $I_{j,k}$, we obtain
\[
I_j
\lesssim
2^{-2jN}
\left(
\int_0^{r^2}
\|f_j(s)\|_{L^2}^2s^\beta ds
\right)^{1/2}.
\]
As $f_j$ is supported in
$B(x_0,2^{j+1}r)$, the definition of the tent-space norm yields
\[
\begin{aligned}
I_j
&\lesssim
2^{-2jN}
(2^{j+1}r)^{\frac{n+2\beta-4}{2}}
\|f\|_{T^{\infty,2}(t^\beta,dy,dt)} \\
&\lesssim
2^{-j\left(2N-\frac{n+2\beta-4}{2}\right)}
r^{\frac{n+2\beta-4}{2}}
\|f\|_{T^{\infty,2}(t^\beta dy dt)}.
\end{aligned}
\]
Choosing
\[
2N>\frac{n+2\beta-4}{2}
\]
ensures summability with respect to $j$. Together with the estimate of $I_0$, this gives
\[
I
\lesssim
r^{\frac{n+2\beta-4}{2}}
\|f\|_{T^{\infty,2}(t^\beta dy dt)},
\]
which completes  the proof of Lemma~\ref{lemdva}.
\end{proof}
\medbreak
Since $v$ defined in \eqref{eq:v} satisfies $\partial_t v=f+\Delta v$, 
 Lemmas \ref{lempiat} and \ref{lemdva} ensure that 
whenever $\alpha>0,$ we have 
\begin{equation}\label{eq:duhatent}
\|(\partial_tv,\nabla^2v)\|_{T^{\infty,2}(t^{1-\alpha}\,dydt)}\lesssim \|f\|_{T^{\infty,2}(t^{1-\alpha}\,dydt)}.\end{equation}
Next, let us establish the boundedness of $v$ in the space $X.$
\begin{lemma}\label{l:X} In any dimension $n$ and for all  $\alpha\in(0,1),$
    the function $v$ defined in \eqref{eq:v} satisfies:
    \begin{align} \label{eq:X1}
        \|v\|_{X}&\lesssim \|f\|_{T^{\infty,2}(t^{1-\alpha}dydt)}.\end{align}    
\end{lemma}
\begin{proof} Since $v(0)=0,$ we can write that
$$ v(t)=\int_0^t\partial_s v(s)\,ds.$$
We deduce that  for all $r>0$ and $x_0\in\R^2,$ setting $B_r:=B(x_0,r),$ we have by
Minkowski and Cauchy-Schwarz inequalities,
\begin{align*}\|v(t)\|_{L^2(B_r)}&\leqslant\int_0^t\|\partial_s v(s)\|_{L^2(B_r)}\,ds\\
&\leqslant\biggl(\int_0^ts^{\alpha-1}\,ds\biggr)^{\frac12}
\biggl(\int_0^t\|\partial_s v(s)\|_{L^2(B_r)}^2 s^{1-\alpha}\,ds\biggr)^{\frac12}\\
&\lesssim t^{\frac\alpha2}\biggl(\int_0^t\|\partial_s v(s)\|_{L^2(B_r)}^2 s^{1-\alpha}\,ds\biggr)^{\frac12}\cdotp
\end{align*}
Multiplying both sides by $r^{1-\frac n2}$ then 
taking the supremum on all $x_0\in\R^n$
and $0<t<r^2,$ we conclude that
\begin{equation}\label{eq:vX}\|v\|_X\lesssim \|\partial_s v\|_{T^{\infty,2}(t^{1-\alpha}\,dydt)}.\end{equation}
As $\partial_s v=f+\Delta v,$ using Lemma \ref{lemdva} completes the proof.\end{proof}
\medbreak
Let us next show that $f\in T^{\infty,2}(t^{1-\alpha}\,dydt)$ implies that $v\in L^\infty(\R_+;\dot B^{-1}_{\infty,\infty}).$
\begin{lemma}\label{l:besov}  Let $\alpha\in(0,1).$ In dimension $1\leqslant n\leqslant4,$ we have 
$$\|v\|_{L^\infty(\R_+;\dot B^{-1}_{\infty,\infty})} \lesssim \|f\|_{T^{\infty,2}(t^{1-\alpha}dy\,dt)}.$$
\end{lemma}
\begin{proof}
In light of \eqref{eq:char},   it is enough to prove that, uniformly in
\(t>0\), \(\sigma>0\), and \(x\in\R^n\), we have
\begin{equation}\label{eq:Besov-Duhamel-target-detailed}
  \sigma^{1/2}\abs{e^{\sigma\Delta}v(t,x)}  \lesssim  \norm{f}_{T^{\infty,2}(t^{1-\alpha}dy\,dt)} .
\end{equation}
Set, for all   $0<s<t$ and $\sigma>0,$ 
 \[ \lambda_s:=t-s+\sigma \andf   R:=(t+\sigma)^{1/2}.\]
We have
\[  e^{\sigma\Delta}v(t,x)=  \int_0^t\int_{\R^3}\Phi_{\lambda_s}(x-y)f(s,y)\,dy\,ds .\]
Recall  the standard  heat kernel bound:
\begin{equation}\label{eq:oseen-bound-detailed}
  \abs{\Phi_\lambda(z)} \lesssim  \lambda^{-n/2}  \left(1+\frac{\abs{z}}{\sqrt\lambda}\right)^{-n},
  \qquad \lambda>0,\ z\in\R^n.
\end{equation}
For fixed $x\in\R^n$ and $t,\sigma>0,$ we decompose $\R^n$ into the union of the ball 
$ A_0:=B(x,R)$ and of the following annuli centered at \(x\):
\[  A_j:=B(x,2^{j+1}R)\setminus B(x,2^jR),\quad j\ge1.\]
We obviously have 
$$ e^{\sigma\Delta}v(t,x)= \sum_{j\geqslant0} I_j \with   I_j:=\int_0^t\int_{A_j}\Phi_{\lambda_s}(x-y)f(s,y)\,dy\,ds.$$
By Cauchy-Schwarz inequality with the weight \(s^{1-\alpha}\), we have
\begin{align}\label{eq:I-j-CS-detailed}
  \abs{I_j} \leqslant F_j B_j\with 
  F_j  &:={\left(\int_0^t\int_{A_j}\abs{f(s,y)}^2s^{1-\alpha}\,dy\,ds\right)^{1/2}},\\\nonumber
 \andf  B_j
  &:={\left(\int_0^t\int_{A_j}\abs{\Phi_{\lambda_s}(x-y)}^2s^{\alpha-1}\,dy\,ds\right)^{1/2}} .
\end{align}
The first factor is controlled by the tent norm.  Indeed, since
\(A_j\subset B(x,2^{j+1}R)\) and \(t\le R^2\le (2^{j+1}R)^2\), we have
\begin{equation}\label{eq:first-factor-detailed}
  F_j   \lesssim  (2^jR)^{(n-2-2\alpha)/2}\norm{f}_{T^{\infty,2}(t^{1-\alpha}dy\,dt)},
  \qquad j\ge0.
\end{equation}
The worst term is the one that corresponds to the ball $A_0$ since $y\mapsto \Phi_{\lambda_s}(x-y)$
is singular at $x.$ 
 Now, from \eqref{eq:oseen-bound-detailed},  we have
\[  \int_{A_0}\abs{\Phi_{\lambda_s}(x-y)}^2\,dy \leqslant \norm{\Phi_{\lambda_s}}_{L^2(\R^3)}^2
  \lesssim   \lambda_s^{-n/2}.\]
Therefore
\[  B_0^2  \lesssim
  \int_0^t (t-s+\sigma)^{-n/2}s^{\alpha-1}\,ds.
\]
To estimate the right-hand side if  \(0<\sigma\le t\), we  split the integral into two pieces.  
For the part pertaining to  \((0,t/2)\), we have 
\[
  \int_0^{t/2}(t-s+\sigma)^{-n/2}s^{\alpha-1}\,ds
  \lesssim  t^{-n/2}\int_0^{t/2}s^{\alpha-1}\,ds
  \lesssim   t^{\alpha-n/2}
  \le   t^{\alpha-1}\sigma^{1-n/2}.
\]
On \((t/2,t)\), since \(s^{\alpha-1}\lesssim t^{\alpha-1}\),
\[  \int_{t/2}^t(t-s+\sigma)^{-n/2}s^{\alpha-1}\,ds\lesssim
  t^{\alpha-1}\int_{t/2}^t(t-s+\sigma)^{-n/2}\,ds  \lesssim  t^{\alpha-1}\sigma^{1-n/2}.\]
Thus, for \(0<\sigma\le t\),
\begin{equation}\label{eq:B0-sigma-small-detailed}
  B_0^2\lesssim t^{\alpha-1}\sigma^{1-n/2}.\end{equation}
Since \(R^2\simeq t\) for $0<\sigma\leqslant t,$  \eqref{eq:first-factor-detailed} and
\eqref{eq:B0-sigma-small-detailed} give
\begin{align*}  \sigma^{1/2}F_0B_0
  &\lesssim  \sigma^{1/2}R^{(n-2-2\alpha)/2}
  t^{(\alpha-1)/2}\sigma^{1/2-n/4}\norm{f}_{T^{\infty,2}(t^{1-\alpha}dy\,dt)}\\
  &\lesssim   \left(\frac{\sigma}{t}\right)^{1-n/4}\norm{f}_{T^{\infty,2}(t^{1-\alpha}dy\,dt)}\\
  &\lesssim  \norm{f}_{T^{\infty,2}(t^{1-\alpha}dy\,dt)}.\end{align*}
If \(\sigma\ge t\), then \(t-s+\sigma\ge\sigma\), whence
\begin{equation}\label{eq:B0-sigma-large-detailed}
  B_0^2   \lesssim  \sigma^{-n/2}\int_0^t s^{\alpha-1}\,ds
  \lesssim  \sigma^{-n/2}t^\alpha.\end{equation}
Since now \(R^2\simeq \sigma\), we get
\begin{align*}   \sigma^{1/2}F_0B_0&\lesssim  \sigma^{1/2}R^{(n-2-2\alpha)/2}
  \sigma^{-n/4}t^{\alpha/2}\norm{f}_{T^{\infty,2}(t^{1-\alpha}dy\,dt)}\\
  &\lesssim  \left(\frac{t}{\sigma}\right)^{\alpha/2}\norm{f}_{T^{\infty,2}(t^{1-\alpha}dy\,dt)}\\
  &\lesssim  \norm{f}_{T^{\infty,2}(t^{1-\alpha}dy\,dt)}.
\end{align*}
Thus, for any $\sigma,t>0,$ we have
\begin{equation}\label{eq:I0-detailed}
  \sigma^{1/2}\abs{I_0}  \lesssim  \norm{f}_{T^{\infty,2}(t^{1-\alpha}dy\,dt)}.
\end{equation}
It remains to estimate the terms that correspond to off-diagonal annuli \(A_j\) for \(j\ge1\).  Now,
\(y\in A_j\) implies that \(\abs{x-y}\gtrsim 2^jR\).  Since
\(\lambda_s=t-s+\sigma\le t+\sigma=R^2\), the kernel bound
\eqref{eq:oseen-bound-detailed} gives
\[  \abs{\Phi_{\lambda_s}(x-y)}\lesssim  \lambda_s^{-n/2}
 \left(\frac{\sqrt{\lambda_s}}{2^jR}\right)^n  \lesssim 2^{-nj}R^{-n}.\]
Consequently,
\[  \int_{A_j}\abs{\Phi_{\lambda_s}(x-y)}^2\,dy\lesssim
  2^{-2nj}R^{-2n}\abs{A_j}  \lesssim 2^{-nj}R^{-n}.\]
Therefore
\begin{equation}\label{eq:Bj-detailed}B_j^2  \lesssim
  2^{-nj}R^{-n}\int_0^t s^{\alpha-1}\,ds  \lesssim  2^{-nj}R^{-n}t^\alpha,\end{equation}
that is,
\[  B_j  \lesssim  2^{-nj/2}R^{-n/2}t^{\alpha/2}.\]
Combining this with \eqref{eq:first-factor-detailed}, we obtain
\begin{align*}  \sigma^{1/2}\abs{I_j}
  &\lesssim  \sigma^{1/2}  (2^jR)^{(n-2-2\alpha)/2}  2^{-nj/2}R^{-n/2}t^{\alpha/2}
  \norm{f}_{T^{\infty,2}(t^{1-\alpha}dy\,dt)}\\  &=  2^{-j(1+\alpha)}
  \bigl(\sigma^{1/2}R^{-1-\alpha}t^{\alpha/2}\bigr)  \norm{f}_{T^{\infty,2}(t^{1-\alpha}dy\,dt)}.
\end{align*}
The factor in parentheses is uniformly bounded. Indeed, if \(0<\sigma\le t\), then
\(R^2\simeq t\), and
\[  \sigma^{1/2}R^{-1-\alpha}t^{\alpha/2}\lesssim
  \left(\frac{\sigma}{t}\right)^{1/2}  \le1.\]
If \(\sigma\ge t\), then \(R^2\simeq\sigma\), and
\[  \sigma^{1/2}R^{-1-\alpha}t^{\alpha/2}\lesssim \left(\frac{t}{\sigma}\right)^{\alpha/2} \le1.\]
Hence, for every \(j\ge1\),
\begin{equation}\label{eq:Ij-offdiag-detailed}
  \sigma^{1/2}\abs{I_j}\lesssim
  2^{-j(1+\alpha)}\norm{f}_{T^{\infty,2}(t^{1-\alpha}dy\,dt)}.\end{equation}
The series \(\sum_{j\ge1}2^{-j(1+\alpha)}\) is summable.  Hence, combining \eqref{eq:I0-detailed} and \eqref{eq:Ij-offdiag-detailed}, we obtain
\eqref{eq:Besov-Duhamel-target-detailed}.  Taking the supremum over
\(x\in\R^n\), \(\sigma>0\), and \(t>0\)   completes the proof.
\end{proof}

Finally, bounding the gradient of $v$ is ensured by the following lemma.
\begin{lemma}\label{l:nabla} If $0<\alpha<1$ then we have 
$$\|\nabla v\|_{\mathbf{T}^{\infty,2}(t^{-\alpha}dy\,dt)}\lesssim \|f\|_{{T}^{\infty,2}(t^{1-\alpha}dy\,dt)}.$$    
\end{lemma}
\begin{proof}  Denote by $T$ the map $f\mapsto \nabla v$ with $v$ defined in \eqref{eq:v}.
For the time being, fix some $x_0\in\R^n$ and $r>0.$ 
We want to show (using again the short notation $B_r=B(x_0,r)$) that
$$\int_0^{r^2}\!\!\int_{B_r}|Tf(y,t)|^2t^{-\alpha}\,dy\,dt
\lesssim r^{n-2\alpha-2}\|f\|^2_{T^{\infty,2}(t^{1-\alpha}dy dt)}.$$
To do this, set $F_0:=B(x_0,2r)$  and  $F_j:=B(x_0,2^{j+1}r)\setminus B(x_0,2^j r)$ for $j\geqslant1,$ and decompose 
$f$ into 
$$f=f_0+\sum_{j\geqslant1} f_j\with f_j(x,t):= f(x,t)\1_{F_j}(x)\ \hbox{ for all } \  j\geqslant0.$$
To handle the term $f_0,$ it suffices to use the well-known property 
$$\|\nabla e^{\tau\Delta}\|_{{\mathcal L}(L^2(\R^n))}\lesssim \tau^{-\frac12},\qquad \tau>0.$$
From it and Minkowski inequality, we readily get 
$$\|T f_0(t)\|_{L^2(B_r)}\leqslant\int_0^t \|\nabla e^{(t-s)\Delta}f_0(s)\|_{L^2(\R^n)}\,ds
\lesssim \int_0^t (t-s)^{-\frac12}\|f_0(s)\|_{L^2}\,ds.$$
Then, using the weighted Hardy inequality of Lemma \ref{l:hardy} (in the appendix), we end up with
$$\int_0^{r^2}\|Tf_0(t)\|^2_{L^2(B_r)}t^{-\alpha}\,dt \lesssim \int_0^{r^2} \|f_0(s)\|_{L^2(B_r)}^2 s^{1-\alpha}\,ds
\lesssim r^{n-2\alpha-2}\|f\|^2_{T^{\infty,2}(t^{1-\alpha}dy dt)}.$$
To bound the terms $Tf_j$ with $j\geqslant1,$ we note that by Minkowski inequality, 
\begin{align*}
\int_0^{r^2} \|Tf_j(t)\|^2_{L^2(B_r)} t^{-\alpha}\,dt &\leqslant I_j +\sum_{k\geqslant1} J_{j,k}\\ 
\with{I}_j&:= \int_0^{r^2}\!\!\int_{B_r}\biggl|\int_{\frac t2}^t \nabla e^{(t-s)\Delta}f_j(y,s)\,ds\biggr|^2 t^{-\alpha}\,dt\,dy\\
\andf J_{j,k}&:=\int_{B(x_0,r)}\int_0^{r^2}\biggl|\int_{2^{-k-1}t}^{2^{-k}t}\nabla e^{(t-s)\Delta}f_j(y,s) ds\biggr|^2t^{-\alpha}\, dt\, dy.\end{align*}
To bound the terms $J_{j,k},$ we start from the Cauchy-Schwarz inequality that gives
\begin{equation}\label{eq:Jjk}J_{j,k}\lesssim \int_0^{r^2}\int_{2^{-k-1}t}^{2^{-k}t}\biggl(\int_{B(x_0,r)}\left|\nabla e^{(t-s)\Delta}f_j(y,s)\right|^2 dy\biggr)2^{-k} t^{1-\alpha}\,ds\, dt.\end{equation}
We argue exactly in the same manner as in Lemma~\ref{lemtri} to write
$$\begin{aligned}
\int_{B(x_0,r)}
\bigg|\nabla e^{(t-s)\Delta}&f_j(y,s)\bigg|^2 dy\\
    &\lesssim   \int_{B(x_0,r)} \biggl|\int_{F_j}(t-s)^{-\frac{n}{2}}\exp\left(-\frac{|x-y|^2}{t-s}\right)\frac{|x-y|}{t-s}f(x,s)\, dx\biggr|^2 dy \\
   &   \lesssim   \int_{B(x_0,r)} \biggl|\int_{F_j}(t-s)^{-\frac{n}{2}}\left(\frac{|x-y|^2}{t-s}\right)^{-\frac{n}{2}-\frac{1}{2}}\frac{|x-y|}{t-s}f(x,s) \,dx\biggr|^2 dy \\
 &\lesssim \int_{B(x_0,r)} \biggl|\int_{F_j}\frac{f(x,s)}{|x-y|^n\sqrt{t-s}} dx \biggr|^2 dy\\[1ex]
 &\lesssim (2^jr)^{-2n} r^n(2^jr)^n \|f(\cdot,s)\|_{L^2(F_j)}^2(t-s)^{-1}\\[1ex]
& \lesssim 2^{-jn} \|f(\cdot,s)\|_{L^2(F_j)}^2 (t-s)^{-1}.
 \end{aligned}$$
Plugging this into \eqref{eq:Jjk} yields
$$\begin{aligned}
    J_{j,k}&\lesssim  \int_0^{r^2}\int_{2^{-k-1}t}^{2^{-k}t} 2^{-jn} (t-s)^{-1} \|f(\cdot,s)\|_{L^2(F_j)}^2 2^{-k} t^{1-\alpha}\,ds\, dt\\
    &=\int_0^{2^{-k}r^2}\int_{2^{k}s}^{2^{k+1}s} 2^{-jn} (t-s)^{-1} \|f(\cdot,s)\|_{L^2(F_j)}^2 2^{-k} t^{1-\alpha}\,dt\, ds \\
    &\lesssim  \int_0^{2^{-k}r^2} \|f(\cdot,s)\|_{L^2(F_j)}^2 2^{-jn} 2^{-k} (2^ks)^{-1} (2^ks)^{2-\alpha}\,ds \\
    &\lesssim 2^{-jn} 2^{-k\alpha} \int_0^{(2^jr)^2}\int_{B(x_0,2^jr)}s^{1-\alpha} f^2(y,s) \,dy\, ds\\
    &\lesssim 2^{-k\alpha}2^{-j(2\alpha+2)} r^{n-2\alpha-2}\|f\|_{T^{\infty,2}(t^{1-\alpha}\,dy\,dt)},
\end{aligned}$$
where the equality in the second line follows from an easy change of variable.
\smallbreak
Minkowski inequality and summation of $J_{j,k}$ over all $j,k\geqslant 1 $ yield
$$\int_0^{r^2}\!\!\int_{B_r}\biggl|\int_0^{\frac t2} \nabla e^{(t-s)\Delta} (f\1_{\R^n\setminus B_r})(y,s)\,ds\biggr|^2 t^{-\alpha}\,dt\,dy\leqslant r^{n-2\alpha-2}\|f\|^2_{T^{\infty,2}(t^{1-\alpha}dy dt)}.$$
All that is left for us to do is to bound the terms $I_j$ for all $j\geqslant1.$
To do this, one can take advantage of the off-diagonal estimate
\eqref{eq:offdiag2} in the appendix. 
Inserting it into the inner integrals that define $I_j$ yields
$$I_j\lesssim \int_0^{r^2} \int_{\frac t2}^t (t-s)^{-1} \Bigl( 1+ \frac{(2^j r)^2}{\tau} \Bigr)^{-2N}
\|f_j(s)\|_{L^2}^2 t^{1-\alpha}\,ds\,dt.$$
Hence, swapping the order of integration, we discover that 
$$I_j\lesssim \int_0^{r^2} \|f_j(s)\|_{L^2}^2 s^{1-\alpha}\biggl(\int_0^s u^{-1}\Bigl(1+\frac{4^jr^2}{u}\Bigr)^{-2N}du\biggr)ds.
$$
Now, for $s\leqslant r^2,$ it is clear that 
$$\int_0^s u^{-1}\Bigl(1+\frac{4^jr^2}{u}\Bigr)^{-2N}du\lesssim 4^{-4jN},
$$
which eventually leads (due to the definition of $f_j$) to 
$$I_j\lesssim 2^{-4jN} \int_0^{r^2} \|f_j(s)\|_{L^2}^2 s^{1-\alpha}\,ds\lesssim 
2^{-4jN} (2^{j+1}r)^{n-2\alpha-2}\|f\|^2_{T^{\infty,2}(t^{1-\alpha}dydt)}.$$
If one takes $N$ sufficiently large, then the right-hand side is summable with respect to $j,$ 
which completes the proof. 
\end{proof}

%%%%%%%%%%%%%%%%%%%%%%%%%%%%%%%%%%%%%%%%

\subsection{Estimate of the convection term}

The general idea is that, since $u\cdot\nabla u=\div(u\otimes u),$  the following 
bilinear estimate holds true:
\begin{equation}
\|u\cdot\nabla u\|_{L^2(\R^n)}\lesssim \|u\|_{\dot B^{-1}_{\infty,\infty}(\R^n)}\|\nabla^2 u\|_{L^2(\R^n)}.\end{equation}
This can be easily obtained from Bony's decomposition and continuity results of the paraproduct and remainder operators
(see below). 
The difficulty however, is that, to be consistent with  our definition of tent spaces, we need to establish an
inequality involving the $L^2$ norm \emph{on balls}, rather than on the whole space. 
To do this, we shall first establish two preliminary results. 

\begin{lemma}\label{l:locH2}   Let $x_0\in\R^n$ and $r>0.$ Denote $B_r:=B(x_0,r).$ 
There exists an absolute constant such that for any $u\in H^2_{loc}(\R^n)$ and $\phi\in\cC^\infty_c(B_{2r})$ 
such that $\phi\equiv1$ on a neigborhood of $B_r,$ we have 
$$
\|u\phi\|_{\dot H^2}\leqslant C\bigl(\|u\|_{L^2(B_{2r})}\|D\phi\|_{L^\infty}^{1/2}\|D^3\phi\|^{1/2}_{L^\infty}+ \|D^2u\|_{L^2(B_{2r})}\|\phi\|_{L^\infty}\bigr)\cdotp$$
\end{lemma}
\begin{proof}
Due to the properties of $\phi,$ we have
$$
\|\nabla^2u\|_{L^2(B_r)}=\|\nabla^2(\phi u)\|_{L^2(B_r)}\leqslant \|\nabla^2(\phi u)\|_{L^2(B_{2r})}=\|\Delta(\phi u)\|_{L^2(B_{2r})}.$$
By Leibniz formula, 
$$\Delta(\phi u)=u\Delta\phi +\phi\Delta u+2\nabla u\cdot\nabla\phi.$$
Hence 
$$\|\nabla^2u\|_{L^2(B_{2r})}\leq\|u\|_{L^2(B_{2r})}\|\Delta\phi\|_{L^\infty} +  \|\Delta u\|_{L^2(B_{2r})}\|\phi\|_{L^\infty} 
+2\|\nabla u\cdot\nabla\phi\|_{L^2(B_{2r})}.$$
Since $\phi$ and $\nabla \phi$ vanish on $\partial B_{2r},$ integrating by parts reveals that (with the summation convention on repeated indices):
\begin{align*}
\|\nabla u\cdot\nabla\phi\|_{L^2(B_{2r})}^2&\leqslant \int_{B_{2r}} \partial_j u\partial_j\phi \partial_ku\partial_k\phi\,dx\\
&=-\int_{B_{2r}} \!\! u \partial^2_{jk}u\partial_j\phi\partial_k\phi\,dx
-\int_{B_{2r}}\!\! u\Delta\phi \nabla u\cdot\nabla\phi\,dx+\frac14\int_{B_{2r}}\! u^2\Delta|\nabla\phi|^2dx.\end{align*}
Hence
$$\displaylines{\|\nabla u\cdot\nabla\phi\|_{L^2(B_{2r})}^2\leqslant \|u\|_{L^2(B_{2r})}\|D^2u\|_{L^2(B_{2r})}\|\nabla\phi\|_{L^\infty}^2+
\|u\|_{L^2(B_{2r})}\|\Delta\phi\|_{L^\infty}\|\nabla u\cdot\nabla\phi\|_{L^2(B_{2r})}\hfill\cr\hfill+\|u\|_{L^2(B_{2r})}^2(\|D^2\phi\|_{L^\infty}^2
+\|D\phi\|_{L^\infty}\|D^3\phi\|_{L^\infty}).}$$
By Young inequality and the fact that 
$$\|\nabla\phi\|_{L^\infty}^2\leq\|\phi\|_{L^\infty}\|\nabla^2\phi\|_{L^\infty}\andf
\|D^2\phi\|_{L^\infty}^2\leqslant \|D\phi\|_{L^\infty}\|D^3\phi\|_{L^\infty},$$ we deduce that
$$\|\nabla u\cdot\nabla\phi\|_{L^2(B_{2r})}\lesssim \|u\|_{L^2(B_{2r})}\|D\phi\|_{L^\infty}^{1/2}\|D^3\phi\|_{L^\infty}^{1/2}+ \|D^2u\|_{L^2(B_{2r})}\|\phi\|_{L^\infty},$$
whence the desired inequality.
\end{proof}

\begin{lemma}\label{l:loc} The space $\dot B^{-1}_{\infty,\infty}(\R^n)$ is stable by multiplication by $C^\infty_c(\R^n)$ functions. More precisely, 
in the case $n\geq3,$ if $\chi$ is in $\cC^\infty_c(\R^n)$ and $u,$ in $\dot B^{-1}_{\infty,\infty}(\R^n)$  then the following inequality holds:
$$\|\chi u\|_{\dot B^{-1}_{\infty,\infty}(\R^n)}\leqslant 
C\|\chi\|_{L^\infty(\R^n)\cap \dot H^{n/2}(\R^n)} \|u\|_{\dot B^{-1}_{\infty,\infty}(\R^n)}.$$
\end{lemma}
\begin{proof} It relies on Bony's decomposition (see \cite[Chap. 2]{BCD}):
$$\chi u=T_\chi u+ T_u\chi +R(u,\chi)$$
and  the following well-known facts:
\begin{enumerate}
\item the paraproduct operator $T$ maps $L^\infty\times \dot B^{-1}_{\infty,\infty}$ to $ \dot B^{-1}_{\infty,\infty};$
\item  the paraproduct operator $T$ maps  $\dot B^{-1}_{\infty,\infty}\times L^\infty$ to $ \dot B^{-1}_{\infty,\infty};$
\item   the remainder operator $R$ maps $\dot H^{n/2}\times \dot B^{-1}_{\infty,\infty}$ to $ \dot B^{-1}_{\infty,\infty};$
\end{enumerate} 
where, we used that $n\geq3$ in the last item.
\end{proof}

\begin{lemma}\label{l:nonlinear}
The following inequality holds true:
\begin{equation}\label{eq:NL2}\|u\cdot\nabla u\|_{T^{\infty,2}(t^{1-\alpha}\,dydt)}\lesssim \|u\|_{L^\infty(\R_+;\dot B^{-1}_{\infty,\infty})}
\bigl(\|u\|_X+ \|\Delta u\|_{T^{\infty,2}(t^{1-\alpha}\,dydt)}\bigr)\cdotp\end{equation}
\end{lemma}
\begin{proof}
Fix a function $\chi$ in $\cC^\infty(\R^n),$ supported in the ball $B(0,2)$ and with value $1$ on a neighborhood of $B(0,1).$
For any $x_0\in\R^n$ and $r>0,$ we set 
$$\chi_{x_0,r}(x):= \chi\biggl(\frac{x_0-x}r\biggr),\quad x\in\R^n.$$
Since, by construction, $\chi_{x_0,r}\equiv1$ near $B_r:=B(x_0,r),$ we have
$$u\cdot\nabla u= \chi_{x_0,r} u\cdot \nabla (\chi_{x_0,r} u)\quad\hbox{on}\quad  B(x_0,r),$$ whence
$$\|u\cdot\nabla u\|_{L^2(B_r)} \leqslant \|z\cdot \nabla z\|_{L^2(\R^n)}\with z:=\chi_{x_0,r} u.$$
 We claim that 
\begin{multline}\label{eq:prod1} 
\|z\cdot \nabla z\|_{L^2(B_r)}\lesssim  
r  \|\chi_{x_0,r} u\|_{\dot B^{-1}_{\infty,\infty}}\|u\cdot\nabla\chi\|_{\dot H^2}
\\+\bigl(\|\chi_{x_0,r} u \|_{\dot B^{-1}_{\infty,\infty}}+r\|u\cdot\nabla\chi_{x_0,r}\|_{\dot B^{-1}_{\infty,\infty}}\bigr)\|\chi_{x_0,r} u\|_{\dot H^2}.
\end{multline}
Indeed,  due to $\div u=0,$ we have 
$$z\cdot \nabla z= \div (z\otimes z) - z\: u\cdot\nabla\chi_{x_0,r}.$$
Hence 
\begin{equation}\label{eq:prod2}\|z\cdot \nabla z\|_{L^2(B_r)}\leqslant \|\div (z\otimes z)\|_{L^2(B_r)}+\| z\: u\cdot\nabla\chi\|_{L^2(B_r)}.\end{equation}
For the first term, we use that, due to the definition of $z,$ of $\chi_{x_0,r}\equiv1$ on $B_r$ and to product laws in Besov spaces, 
\begin{equation}\label{eq:divdiv}\|\div (z\otimes z)\|_{L^2(B_r)}\leqslant  \|\div (z\otimes z)\|_{L^2(\R^n)}\leqslant \|z\otimes z\|_{\dot H^1}\lesssim \|z\|_{\dot B^{-1}_{\infty,\infty}}
 \|z\|_{\dot H^2}.\end{equation}
 This stems from the fact that any product $z^i z^j$ may be decomposed according to Bony's decomposition (see \cite[Chap. 2]{BCD})
 $$ z^iz^j=T_{z^i}z^j +T_{z^j}z^i +R(z^i,z^j)$$ and that
 \begin{enumerate}
 \item  the paraproduct operator maps  $\dot B^{-1}_{\infty,\infty}\times \dot H^2$ to $ \dot H^{1};$
\item   the remainder operator maps  $\dot B^{-1}_{\infty,\infty}\times \dot H^2$ to $ \dot H^{1}.$
\end{enumerate}  
The second term of \eqref{eq:prod2} is supported in $B_{2r}.$ Hence, owing to Poincar\'e's inequality, 
$$\| z\: u\cdot\nabla\chi_{x_0,r}\|_{L^2(B_r)}\lesssim r \| z\: u\cdot\nabla\chi_{x_0,r}\|_{\dot H^1}.$$
Now, the continuity results for the paraproduct and remainder we mentioned just above ensure that 
$$\| z\: u\cdot\nabla\chi_{x_0,r}\|_{\dot H^1}\lesssim \|z\|_{\dot B^{-1}_{\infty,\infty}}\|u\cdot\nabla\chi_{x_0,r}\|_{\dot H^2}
+  \|u\cdot\nabla\chi_{x_0,r}\|_{\dot B^{-1}_{\infty,\infty}}\|z\|_{\dot H^2},$$
whence \eqref{eq:prod1}.
\medbreak
Clearly, owing to the definition of $\chi_{x_0,r}$ from $\chi$ by dilation and translation, we have
$$
\|\chi_{x_0,r}\|_{L^\infty}+\|\chi_{x_0,r}\|_{\dot H^{n/2}}+r\|\nabla\chi_{x_0,r}\|_{\dot H^{n/2}}\simeq 1\andf 
\|\nabla^k\chi_{x_0,r}\|_{L^\infty}\lesssim r^{-k}\quad\hbox{for }\  k\in\N.$$
Consequently, Lemma \ref{l:loc} ensures that
$$\|z\|_{\dot B^{-1}_{\infty,\infty}}+r\|u\cdot\nabla\chi_{x_0,r}\|_{\dot B^{-1}_{\infty,\infty}}\lesssim \|u \|_{\dot B^{-1}_{\infty,\infty}}$$
while Lemma \ref{l:locH2} gives us 
$$r\|u\cdot\nabla\chi_{x_0,r}\|_{\dot H^2} + \|z\|_{\dot H^2}\lesssim r^{-2}\|u\|_{L^2(B_{2r})} + \|\nabla^2u\|_{L^2(B_{2r})}.$$
We conclude that there exists an absolute constant $C$ such that for all $t>0,$ $x_0\in\R^n$ and $r>0,$ we have
$$\|(u\cdot\nabla u)(t)\|_{L^2(B_r)}\leqslant C\|u(t)\|_{\dot B^{-1}_{\infty,\infty}}\bigl( r^{-2}\|u(t)\|_{L^2(B_{2r})} 
+ \|\nabla^2u(t)\|_{L^2(B_{2r})}\bigr)\cdotp$$
At this stage, taking the $L^2([0,r^2], t^{1-\alpha}dt)$ norm of both sides, multiplying by $r^{\alpha+1-n/2},$ then 
taking the supremum on $x_0\in\R^n$ and $r>0,$ and recalling the definition of $X,$ we get 
the desired inequality. \end{proof}

%%%%%%%%%%%%%%%%%%%%%%%%%%%%%%%%%%%%%%%%%%%%%%%%

\section{The proof of existence}\label{s:existence}

The strategy is to smooth out the initial velocity and to solve \eqref{inhom}
with the corresponding data by taking advantage of the classical well-posedness theory,
then to prove that the resulting sequence $(\rho^m,u^m)_{m\in\N}$ of solutions exists for all positive times
and is bounded in $L^\infty\times E_\alpha.$  This will enable us
 to pass to the limit in \eqref{inhom}, up to subsequence, by means of standard compactness arguments. 
\medbreak
More precisely, after fixing an initial data $(\rho_0,u_0)$ that meets the conditions of Theorem \ref{key},
we consider a sequence $(u_0^m)_{m\in\N}$ of divergence-free vector fields with coefficients in 
${H}^1$ such that  
\begin{equation}
\label{eq:convu}
\sup_{m\in\N} \|u_0^m\|_{U_\alpha}\lesssim \|u_0\|_{U_\alpha}\andf
\lim_{m\to\infty} u_0^m=u_0\quad\hbox{in}\quad {\mathcal S}'(\R^3).\end{equation}
Since the space $U_\alpha$ is defined by completion, one can find 
a sequence $(v_0^m)_{m\in\N}$ of Schwartz functions converging to $u_0$ in $U_\alpha.$
Then, setting $u_0^m:={\mathbb P}v_0^m,$ the elements of the sequence are in all Sobolev spaces and still converge to $u_0$ since ${\mathbb P}u_0=u_0$. Furthermore, Lemma \ref{l} ensures the inequality in \eqref{eq:convu}.

\medbreak
As $\rho_0$ is bounded and bounded away from zero, and $u_0^m$ is in ${H}^1,$ one can apply\footnote{There, the result is stated in the torus or 
 in a bounded domain, but adaptation to $\R^3$ is easy, see e.g. \cite[Rem. 2.1]{PZZ}.}
  \cite[Rem. 2.4]{dan14} to get a  unique maximal solution $(\rho^m,u^m,\nabla p^m)$ of \eqref{inhom} with  data $(\rho_0,u_0^m)$ 
   defined on $\R^3\times[0,T^m),$ 
satisfying   \begin{equation}\label{eq:boundrhom}
\forall t\in[0,T^m),\; \|a^m(t)\|_{L^\infty}=\|a_0\|_{L^\infty}\with a^m:=\rho^m-1,\end{equation}
  the energy balance  \begin{equation}\label{eq:energy}
\frac 12\|\sqrt{\rho^m(t)}\,u^m(t)\|_{L^2}^2+\int_0^t\|\nabla u^m\|_{L^2}^2\,ds
=\frac 12\|\sqrt{\rho_0}\,u^m_0\|_{L^2}^2,\qquad t\in[0,T^m),\end{equation}
and  such that for all $T<T^m,$ we have
  $$u^m\in L^\infty([0,T];H^1(\R^3)) \cap L^2(0,T;H^2(\R^3))\andf\partial_tu^m,\nabla p^m \in L^2(0,T;L^2(\R^3)).$$
We claim that there exist two  constants $c$ and $C$ depending only on $\alpha\in(0,1)$ and such that, if
\begin{equation}\label{eq:smalla0}
    \|a_0\|_{L^\infty} \leqslant c,
\end{equation}
then, denoting by $E_\alpha(T)$ the version of $E_\alpha$ pertaining to $\R^3\times(0,T),$  we have
\begin{equation}\label{eq:global}
    \|u^m\|_{E_\alpha(T)} \leqslant C\bigl(\|u_0\|_{U_\alpha}+    \|u^m\|_{E_\alpha(T)}^2\bigr)\quad\hbox{for all }\ T\in(0,T^m).
\end{equation}
Note that the smoothness of $u^m$ ensures that
$$\partial_t(\rho^m u^m) +\div(\rho^m u^m\otimes u^m)=(1+a^m)\partial_tu^m +\rho^m u^m\cdot\nabla u^m.$$ 
Hence $u^m$ satisfies
$$\begin{aligned}u^m(t)&=e^{t\Delta} u_0^m
-\int_0^te^{(t-\tau)\Delta}{\mathbb P}\Bigl((\rho^m u^m\cdot\nabla u^m)(\tau)+
(a^m\partial_\tau u^m)(\tau)\Bigr)d\tau\\&=: u^m_L(t)+v^m(t).\end{aligned}$$
 Lemma \ref{lemodin} readily gives that 
 $\|u_L^m\|_{E_\alpha}\leqslant C \|u_0\|_{U_\alpha}.$
Next, let us split $v^m$ into the part $v^m_1$ corresponding to $\rho^m u^m\cdot\nabla u^m$
and the part $v^m_2$ corresponding to $a^m\partial_\tau u^m.$

For $v^m_1,$ combining Lemma \ref{lempiat},  Inequality  \eqref{eq:duhatent} and Lemmas \ref{l:X}, \ref{l:besov}, 
\ref{l:nabla}, \ref{l:nonlinear}, we discover that for all $T\in(0,T^m),$  
\begin{align*}     \|v^m_1\|_{E_\alpha(T)}
    &\lesssim 
    \|{\mathbb P}(\rho^m u^m\cdot\nabla u^m)\|_{T^{\infty,2}(0,T;t^{1-\alpha}dydt)}\\
  &\lesssim \|\rho^m\|_{L^\infty} \|u^m\cdot\nabla u^m\|_{T^{\infty,2}(0,T;t^{1-\alpha}dydt)}\\
  &\lesssim \|\rho_0\|_{L^\infty}
    \|u^m\|_{L^\infty(0,T;\dot B^{-1}_{\infty,\infty})}\bigl(\|u^m\|_{L^\infty(0,T;\dot B^{-1}_{\infty,\infty})}+ \|\Delta u^m\|_{T^{\infty,2}(0,T;t^{1-\alpha}dydt)}\bigr)\cdotp
\end{align*}
For $v^m_2,$  combining Lemmas \ref{lempiat} and \ref{lemdva}, then \eqref{eq:boundrhom} gives us
\begin{align*}
      \|v^m_2\|_{E_\alpha(T)}
    &\lesssim \|{\mathbb P}(a^m\partial_tu^m)\|_{T^{\infty,2}(0,T;t^{1-\alpha}dydt)}\\
    &\lesssim \|a_0\|_{L^\infty}\| \partial_tu^m\|_{T^{\infty,2}(0,T;t^{1-\alpha}dydt)}.
    \end{align*}
Hence, Lemmas \ref{l:X} and \ref{l:nabla} allow us to get (arguing as before 
to bound the right-hand side):
\begin{multline*}  \|v^m\|_{E_\alpha(T)}\lesssim \|\rho_0\|_{L^\infty}   \|u^m\|_{L^\infty(0,T;\dot B^{-1}_{\infty,\infty})}\bigl(\|u^m\|_{L^\infty(0,T;\dot B^{-1}_{\infty,\infty})}+ \|\Delta u^m\|_{T^{\infty,2}(0,T;t^{1-\alpha}dydt)}\bigr)\\+
    \|a_0\|_{L^\infty}\| \partial_tu^m\|_{T^{\infty,2}(0,T;t^{1-\alpha}\,dydt)}.\end{multline*}
Putting all these inequalities together and assuming that $c$ in \eqref{eq:smalla0}
is small enough completes the proof of \eqref{eq:global}.
Consequently, if 
$$\|u_0\|_{U_\alpha}\leqslant \frac1{4C^2},$$ then we have
 \begin{equation}\label{eq:uniformbound} 
 \sup_{m\in\N}  \|u^m\|_{E_\alpha(T^m)}\leqslant 2C\|u_0\|_{U_\alpha}.\end{equation}
 \medbreak
  Since the $L^2$ norm of $u^m(t)$ can be controlled by the data thanks to the energy balance \eqref{eq:energy}, 
to prove that $T^m=\infty,$  it suffices to check that \eqref{eq:uniformbound} implies that
\begin{equation}\label{eq:unifH1} \sup_{t\in[0,T^m)} \|\nabla u^m(t)\|_{L^2}<\infty.\end{equation}
Indeed,  the uniqueness of the solution $(\rho^m,u^m)$ combined with the lower bound 
(in terms of the $H^1$ norm) given by  \cite[Rem. 2.4]{dan14} would enable us to continue the solution if $T^m$ were finite. 

To achieve \eqref{eq:unifH1}, we use Inequality (3.6) of \cite{dan14} that
is valid in any  space dimension, namely
\begin{equation}\label{eq:H1}\frac d{dt}\|\nabla u^m\|_{L^2}^2 +\frac12\|\sqrt\rho\, \partial_tu^m\|_{L^2}^2
+\frac1{4\|\rho_0\|_{L^\infty}}\|\nabla^2 u^m\|_{L^2}^2
\leqslant\frac32\|u^m\cdot\nabla u^m\|_{L^2}^2.\end{equation}
To bound the right-hand side,  we use the fact that
 $u^m\cdot\nabla u^m=\div(u^m\otimes u^m)$ and
 the following bilinear estimate (see \eqref{eq:divdiv}):
\begin{equation}\label{eq:bilinear}\|ab\|_{\dot H^1}\leqslant C\bigl(\|a\|_{\dot B^{-1}_{\infty,\infty}}\|b\|_{\dot H^2}+\|b\|_{\dot B^{-1}_{\infty,\infty}}\|a\|_{\dot H^2}\bigr)\cdotp\end{equation}
%that can  be  proved by using the  \emph{Bony decomposition} $$ab=T_ab+T_ba +R(a,b).$$
%Indeed, the paraproduct and remainder operators $T$ and $R$ map 
%$\dot B^{-1}_{\infty,\infty}\times \dot H^2$ to $\dot H^1$(see more details in \cite{BCD}).
Hence we have
$$\|u^m\cdot\nabla u^m\|_{L^2}=\|\div(u^m\otimes u^m)\|_{L^2}
\leqslant C \|u^m\|_{\dot B^{-1}_{\infty,\infty}}\|u^m\|_{\dot H^2}.$$
Therefore, using \eqref{eq:uniformbound}  and \eqref{eq:smalldata} yields
$$\|u^m\cdot\nabla u^m\|_{L^2}\leqslant C\varepsilon_0\|u^m\|_{\dot H^{2}}.$$
This term may thus be absorbed by the left-hand side of  \eqref{eq:H1} (taking smaller $\varepsilon_0$ if necessary) and
 integrating in time yields:
 $$
\|\nabla u^m(t)\|_{L^2}^2 + \frac18\int_0^t \bigl(\|\partial_tu^m\|_{L^2}^2+\|\nabla^2 u^m\|_{L^2}^2\bigr)d\tau
\leqslant \|\nabla u^m_0\|_{L^2}^2 \quad\hbox{for all }\ t\in[0,T^m).$$
 Consequently, we have \eqref{eq:unifH1} and thus $T^m=\infty$ as explained before.
 \medbreak
 Granted with \eqref{eq:boundrhom} and \eqref{eq:uniformbound}, we can now assert 
 that $(a^m,u^m)_{m\in\N}$ converges up to subsequence to some pair $(a,u)$  for the weak $*$ topology
 of $L^\infty(\R^n\times\R_+)\times E_\alpha.$
Looking at the boundedness of the time derivative of $(u^m)_{m\in\N},$ one can then glean some strong compactness which 
allows us to pass to the limit in the nonlinear terms of 
 \eqref{inhom}. The rest of the proof is classical. It is not  detailed here for the sake of conciseness.

\subsection*{Acknowledgments}
This work was started while the second named author was granted by the ANR project ANR-15-CE40-0011.
The second named author is also deeply grateful to Alexey Cheskidov, and to Pascal Auscher and Hedong Hou for a number of helpful discussions.

%%%%%%%%%%%%%%%%%%%%%%%%%%%%%%%%%%%%%%%%%%%%%%%%%%%%

\section{Appendix}

The following off-diagonal estimates are the key to Lemmas \ref{lemdva}
and \ref{l:nabla}.

\begin{lemma}
\label{lemtri} Let $x_0\in\R^n$ and $r>0.$
   For $j\geqslant 1,$ set   
  {$F_j:=B(x_0,\,2^{j+1} r)\backslash B(x_0,\, 2^{j}r)$}, 
  and denote $E:=B(x_0,\,r).$   
  For all positive numbers $\theta$ and $N$ and for all functions $f$ in $L^2_{loc}(\mathbb R^n),$ we have 
\begin{align}\label{eq:offdiag1}\|\theta \, \Delta \, e^{\theta \Delta} ( \mathbbm{1}_{F_j} f)\|_{L^2(E)} &\leqslant C \Bigl( 1+ \frac{(2^j r)^2}{\theta} \Bigr)^{-N} \|f\|_{L^2(F_j)}, \\
\label{eq:offdiag2}\|\theta^{\frac12} \nabla \, e^{\theta\Delta}(\1_{F_j}\,f)\|_{L^2(E)} &\leqslant C 
\Bigl( 1+ \frac{(2^j r)^2}{\theta} \Bigr)^{-N} \|f\|_{L^2(F_j)},
\end{align}
where the constant $C$ depends only on $n$ and $N$.
\end{lemma}
\begin{proof}
Let us start with \eqref{eq:offdiag1}.
Due to the fact that the function $e^{\theta \Delta} \bigl( \mathbbm{1}_{F_j} f\bigr)$ solves the system 
\begin{equation*}
 \begin{cases}
   \partial _\theta u = \Delta u, \\
   u \big|_{\theta =0} = \mathbbm{1}_{F_j} f,
 \end{cases}
\end{equation*}
we see that it suffices to estimate the following integral
$$\int_E \left| \left(\theta \, \partial_\theta \Phi_\theta \right) \ast \mathbbm{1}_{F_j} f(x) \right|^2 dx.$$ 
%where $\Phi_\theta$ is the heat kernel in $\mathbb R^n.$ 
To this end, for $z \in \mathbb R^n$, denote
\begin{align*}
K_\theta (z) &:= \theta \cdot \partial_\theta \Phi_\theta (z) 
%= \theta \pi^{-n} \Bigl( - \frac{n}{2} \, \theta ^{- \frac{n}{2}-1} e^{- \frac{|z|^2}{\theta}} +\theta ^{- \frac{n}{2}} e^{- \frac{|z|^2}{\theta}} \frac{|z|^2}{\theta^2} \Bigr) \\&
=\pi^{-n} \theta^{- \frac n2} e^{- \frac{|z|^2}{\theta}} \Bigl(-\frac{n}{2} + \frac{|z|^2}{\theta}\Bigr)\cdotp
\end{align*}
For any $N>0,$  there exists a constant $C$ that depends only on $n$ 
and $N,$ such that 
$$|K_\theta (z) | \leqslant C\, \theta^{-\frac{n}{2}} \Bigl(1 + \frac{|z|^2}{\theta} \Bigr)^{-N-\frac{n}{2}}.$$
Hence, we infer the following estimate:
\begin{align*}
\int_E \big| \bigl( \theta \, \partial_\theta \Psi_\theta &\bigr) \ast \mathbbm{1}_{F_j} f(x) \big|^2 dx \leqslant C\int_E \Big(\int_{F_j} \theta^{-\frac{n}{2}} \Bigl(1+ \frac{|x-y|^2}{\theta} \Bigr)^{-N-\frac{n}{2}} |f(y)| \,dy \Big)^2 dx \\
&\leqslant C \int_E \theta^{-n}\Bigl(1+\frac{(2^j r)^2}{\theta}\Bigr)^{-2N-n}\cdot \Bigl( \int_{F_j} f^2(y) \, dy\Bigr) \cdot \big|{F_j} \big| dx  \\
&\leqslant C \biggl(1+\frac{(2^j r)^2}{\theta}\biggr)^{-2N} \cdot\Bigl( {\frac{\theta}{\theta +(2^j r)^2}}\Bigr)^n\cdot \theta^{-n} \bigl(2^j r\bigr)^n r^n\cdot \int_{F_j} f^2(y) \, dy \\
&\leqslant C \biggl( 1+ \frac{(2^j r)^2}{\theta} \biggr)^{-2N} \cdot\int_{F_j} f^2(y) \, dy,
\end{align*}
and Inequality \eqref{eq:offdiag1} is proved.
\medbreak
Next, let us prove Inequality \eqref{eq:offdiag2}. 
Obviously, we have $\frac{2^jr}{2}\leqslant \mathrm{dist}(E,F_j)$ and the kernel $K_\theta$ of the operator $f\mapsto \mathbbm{1}_{E} \nabla e^{\theta \Delta} (\mathbbm{1}_{F_j} f)$ is given by
$$K_\theta(x,y)= \mathbbm{1}_{E}(x)\mathbbm{1}_{F_j}(y) \nabla \Phi_\theta(x-y).$$
For $x\in E$ and $y\in F_j$ we have $2^{j}r\leqslant 2 |x-y|$. Therefore,
for any $N>0,$ there exists a constant $C$ such that
$$\sup_{x\in E} \int_{F_j} |K_\theta (x,y)| dy \leqslant C \theta^{-\frac{1}{2}} \int_{|z|\geqslant \frac{2^jr}{2}} \theta^{-\frac{n}{2}}e^{-\frac{|z|^2}{2\theta}} dz \leqslant
C\theta^{-\frac{1}{2}} \left(1+\frac{(2^jr)^2}{\theta}\right)^{-N}\cdotp$$
The very same bound holds for 
$$\sup_{x\in F_j} \int_{E} |K_\theta (x,y)| dx.$$
Schur’s lemma then gives
$$\|\nabla e^{\theta \Delta} (\mathbbm{1}_{F_j} f)\|_{L^2(E)}\lesssim  \theta^{-\frac{1}{2}} \left(1+\frac{(2^jr)^2}{\theta}\right)^{-N} \|f\|_{L^2(F_j)},$$
which means that the desired Inequality \eqref{eq:offdiag2} is valid.
\end{proof}
\medbreak
The following weighted Hardy inequality  was used in the proof of Lemma \ref{l:nabla}.
\begin{lemma}\label{l:hardy}
Let $0<\alpha<1$ and $T>0.$ 
There exists a constant $C_\alpha$ such that for any
measurable function $h:[0,T]\to[0,\infty),$ we have
$$\int_0^{T}\biggl(\int_0^t(t-s)^{-\frac12} h(s)\,ds\biggr)^2t^{-\alpha}\,dt
\leqslant C_\alpha \int_0^{T} h^2(s) s^{1-\alpha}\,ds.$$
\end{lemma}
\begin{proof} Set $g(s):=s^{\frac{1-\alpha}2}h(s).$ From the change of variable $s=tr,$ we get 
$$\int_0^t(t-s)^{-\frac12} h(s)\,ds
= t^{\frac\alpha2}\int_0^1(1-r)^{-\frac12} r^{\frac{\alpha-1}2}g(r)\,dr.$$
This implies that
$$t^{-\frac\alpha2} \int_0^t(t-s)^{-\frac12} h(s)\,ds=\int_0^1 K_\alpha(r) g(tr)\,dr 
\with K_\alpha(r):=(1-r)^{-\frac12} r^{\frac{\alpha-1}2},$$
whence,  taking the $L^2(0,T)$ norm with respect to the variable $t,$
\begin{align*}
\biggl\|t^{-\frac\alpha2} \int_0^t(t-s)^{-\frac12} h(s)\,ds\biggr\|_{L^2(0,T)}
&\leqslant \int_0^1 K_\alpha(r) \|g(tr)\|_{L^2(0,T)}\\
&\leqslant \|g\|_{L^2(0,T)} \int_0^1 r^{-\frac12} K_\alpha(r)\,dr\\
&\leqslant \|g\|_{L^2(0,T)}\int_0^1 (1-r)^{-\frac12} r^{\frac\alpha2-1}\,dr\lesssim \|g\|_{L^2(0,T)}.
\end{align*}
Replacing $g$ by its value completes the proof. 
\end{proof}
\medbreak

The following proposition states that the space  $U_{\alpha}$ 
is \emph{critical}. Although the proof is elementary, we give it for the reader's convenience.  
\begin{prop}
\label{prop1}
For all $\lambda >0,$ if $f_\lambda (x) := f(\lambda \, x),$  then 
$\|f_\lambda\|_{ U_{\alpha}} = \lambda ^{-1}  \|f\|_{ U_{\alpha}}$.
\end{prop}
\begin{proof}
Note that $f_\lambda \ast \Psi_t (y)= f \ast \Psi_{\lambda t} (\lambda \, y)$ for all $y \in \mathbb R^n$.
 Hence, using the change of variable $s=\lambda t$ yields
\begin{equation}
\begin{split}
  \|f_\lambda\|^2_{U_{\alpha}}&= \underset{\atop{r>0}{x_0 \in \mathbb R^n}} {\sup} \,  \frac{1}{r^{-n-2 \alpha - 2}} \int_0 ^r \!\!\int_{B(x_0, r)} | f_\lambda \ast \Psi_t (y)|^2 \, t^{-1-2\alpha} \, dy\,dt  \\
&= \underset{\atop{r>0}{x_0 \in \mathbb R^n}}{\sup} \,  \frac{1}{r^{-n-2 \alpha - 2}} \int_0 ^ r \!\!\int_{B(x_0, r)} | f \ast \Psi_{ \lambda t} (\lambda y)|^2 \, t^{-1-2\alpha} \, dy\,dt  \\
&= \underset{\atop{r>0}{x_0 \in \mathbb R^n}}{\sup} \,  \frac{1}{r^{-n-2 \alpha - 2}} \int_0 ^{\lambda r}\!\! \int_{B(\lambda x_0, \, \lambda r)} | f \ast \Psi_s ( z )|^2 \, s^{-1-2\alpha} \, \lambda^{1+2\alpha} \, \lambda^{-n} \, \lambda^{-1} dz\,ds  \\
&= \underset{\atop{r>0}{x_0 \in \mathbb R^n}}{\sup} \,  \frac{\lambda^{-2}}{{(\lambda r)}^{-n-2 \alpha - 2}} \int_0 ^{ \lambda r} \!\!\int_{B(\lambda x_0, \,\lambda r)} | f \ast \Psi_s (z)|^2 \, s^{-1-2\alpha} \, dz\,ds  \\
&= \lambda ^{-2} \|f\|^2_{U_{\alpha}},
\end{split}
\end{equation}
which completes the proof of Proposition~\ref{prop1}.
\end{proof}
\medbreak 
\begin{lemma}\label{l}
Let \(\alpha>-1\). The Riesz transforms \(R_1,\ldots,R_n\) are bounded on \(U_\alpha\). Consequently, the Leray projector \(\mathbb P\) is bounded on \(U_\alpha(\mathbb R^n;\mathbb R^n)\).
\end{lemma}
\begin{proof}
It is enough to establish the result for \(R_1\). 
Let $\Psi$ be a function that defines $U_\alpha,$ and $\Psi_t:=t^{-n}\Psi(t^{-1}\cdot).$
Since \(R_1\) commutes with convolution, we have \(\Psi_t*(R_1f)=R_1(\Psi_tf).\)
Set \(g_t:=\Psi_tf.\)
Fix \(x_0\in\mathbb R^n\) and \(r>0\), and denote
\(B_r:=B(x_0,r),\) $$F_0:=B(x_0,2r)\ \hbox{ and, }\  \hbox{ for }\ j\geqslant 1,\
F_j:=B(x_0,2^{j+1}r)\setminus B(x_0,2^jr).$$
Writing
\[g_t=\sum_{j\geqslant0} 1_{F_j}g_t\]
and using Minkowski's inequality, we get
\[\left(
\int_0^r\int_{B_r}|R_1g_t(y)|^2
\,\frac{dy\,dt}{t^{1+2\alpha}}
\right)^{1/2}
\leqslant\sum_{j\geqslant0}
\left(\int_0^r\int_{B_r}\left|R_1( 1_{F_j}g_t)(y)\right|^2
\,\frac{dy\,dt}{t^{1+2\alpha}}\right)^{1/2}.\]
For \(j=0\), the \(L^2\)-boundedness of \(R_1\) gives
\[\left(\int_0^r\int_{B_r}\left|R_1( 1_{F_0}g_t)\right|^2
\,\frac{dy\,dt}{t^{1+2\alpha}}
\right)^{1/2}\lesssim\left(\int_0^r\int_{B(x_0,2r)}
|g_t|^2\,\frac{dy\,dt}{t^{1+2\alpha}}
\right)^{1/2}.\]
For \(j\geqslant1\), the kernel estimate for the Riesz transform and the Cauchy--Schwarz inequality yield
\[\left\|R_1( 1_{F_j}g_t)\right\|_{L^2(B_r)}\lesssim
2^{-jn/2}\|g_t\|_{L^2(F_j)}.\]
Therefore,
\[\left(\int_0^r\int_{B_r}|R_1g_t(y)|^2\,\frac{dy\,dt}{t^{1+2\alpha}}
\right)^{1/2}\lesssim\sum_{j\geqslant0}2^{-jn/2}
\left(\int_0^r\int_{B(x_0,2^{j+1}r)}|g_t(y)|^2
\,\frac{dy\,dt}{t^{1+2\alpha}}\right)^{1/2}.\]
Since \(r\leqslant2^{j+1}r\), the definition of \(U_\alpha\) implies
\[\begin{aligned}\left(
\int_0^r\int_{B_r}|R_1g_t(y)|^2
\,\frac{dy,dt}{t^{1+2\alpha}}
\right)^{1/2}
&\lesssim\|f\|_{U_\alpha}
\sum_{j\geqslant0}
2^{-jn/2}(2^{j+1}r)^{(n-2\alpha-2)/2}\\
&\lesssim
r^{(n-2\alpha-2)/2}
\|f\|_{U_\alpha}
\sum_{j\geqslant0}2^{-j(\alpha+1)}.\end{aligned}\]
The last series is convergent because \(\alpha>-1\). 
Then, dividing by $r^{(n-2\alpha-2)/2}$ and taking the supremum over \(x_0\) and \(r\) gives
\(\|R_1f\|_{U\alpha}\lesssim\|f\|_{U\alpha}.\)
The same argument applies to all Riesz transforms. Since
\(\mathbb P_{ij}=\delta_{ij}+R_iR_j,\)
the Leray projector is bounded on \(U_\alpha\).
\end{proof}
\medbreak
The following result allows us to justify the definition of the space $U_\alpha$ by completion.
\begin{prop}\label{p:density}
Let $\alpha\in(0,-1+n/2]$ and $f\in\cS(\R^n).$ Then, $\|f\|_{U_\alpha}<\infty.$
\end{prop}
\begin{proof}
Take any function $\Psi$ in $\cS(\R^n)$ with zero average and set $\Psi_t:=t^{-n}\Psi(t^{-1}\cdot)$ for $t>0.$
It is only a matter of bounding uniformly 
$$r^{2\alpha+2-n}\int_0^r\!\!\int_{B(x_0,r)} \bigl|(\Psi_t\star f)(y)|^2 t^{-1-2\alpha}\,dy\,dt,\quad r>0,\quad x_0\in\R^n.$$
    If $r\in(0,1)$ then we take advantage of the fact that
    \begin{equation}\label{eq:1}f\star\Psi_t(x)=\int_{\R^n} \bigl(f(x-tz)-f(z)\bigr)\Psi(z)\,dz,    \end{equation}
and thus
$$\|f\star\Psi_t\|_{L^\infty(\R^n)} \lesssim t\|\nabla f\|_{L^\infty(\R^n)}.$$
Plugging this in the above integrals, we easily get
$$\int_0^r\!\!\int_{B(x_0,r)} \bigl|(\Psi_t\star f)(y)|^2 t^{-1-2\alpha}\,dy\,dt\lesssim r^{n+2-2\alpha}\|\nabla f\|_{L^\infty(\R^n)}^2.$$
Hence, for $r\in(0,1)$ and $x_0\in\R^n,$ we have
$$r^{2\alpha+2-n}\int_0^r\!\!\int_{B(x_0,r)} \bigl|(\Psi_t\star f)(y)|^2 t^{-1-2\alpha}\,dy\,dt\lesssim r^4\|\nabla f\|_{L^\infty(\R^n)}^2\leq\|\nabla f\|_{L^\infty(\R^n)}^2.$$
If $r\geq1,$ then we split the outer integral into the intervals $[0,1]$ and $[1,r].$
To handle the first part, we use \eqref{eq:1} and the mean value formula which imply that
$$\Psi_t\star f(x)=t\int_0^1\!\!\int_{\R^n} \nabla f(x-\theta tz)\cdot z\:\Psi(z)\,dz\,d\theta,$$
from which we get
$$\|\Psi_t\star f\|_{L^2}\lesssim t\|\nabla f\|_{L^2}.$$
Hence, we have 
$$\int_0^1\!\!\int_{B(x_0,r)} \bigl|(\Psi_t\star f)(y)|^2 t^{-1-2\alpha}\,dy\,dt\lesssim
\|\nabla f\|_{L^2}^2 \int_0^1 t^{1-2\alpha}\,dt\lesssim \|\nabla f\|_{L^2}^2.$$
For the part on $[1,r],$ we just use the convolution inequality
$$\|f\star \Psi_t\|_{L^2}\leq \|f\|_{L^2}\|\Psi\|_{L^1}$$ to write that 
$$\int_1^r\!\!\int_{B(x_0,r)} \bigl|(\Psi_t\star f)(y)|^2 t^{-1-2\alpha}\,dy\,dt\lesssim 
\|f\|_{L^2}^2\int_1^r t^{-1-2\alpha}\,dt\lesssim 
\|f\|_{L^2}^2.$$
In the end, as $r^{2\alpha+2-n}\leq1$ for $r\geq1$ and $\alpha\leq -1+n/2,$ we conclude that
$$r^{2\alpha+2-n} \int_0^r\!\!\int_{B(x_0,r)} \bigl|(\Psi_t\star f)(y)|^2 t^{-1-2\alpha}\,dy\,dt\lesssim
\|f\|_{H^1}^2.$$
This completes the proof of the proposition.
\end{proof}

\smallbreak

\renewcommand{\refname}{References}

\end{document}